\documentclass[10pt, a4paper]{amsart}
\usepackage{setspace} 
\usepackage{amssymb}
\usepackage[dvips]{color}
\usepackage{amsmath}
\usepackage{amsthm}
\usepackage[a4paper, margin=4cm]{geometry}
\usepackage{amsfonts}
\usepackage[all]{xypic} 
\usepackage{bbm}
\usepackage{xcolor}

\usepackage[colorinlistoftodos]{todonotes}

\usepackage{hyperref}
\usepackage{cleveref}
\usepackage{url}

\allowdisplaybreaks[4] 
\newtheoremstyle{myremark}     {10pt}{10pt}{}{}{\bfseries}{.}{.5em}{}

\newtheorem{thm}{Theorem}[section]

\newtheorem{lem}[thm]{Lemma}
\newtheorem{prop}[thm]{Proposition}
\theoremstyle{definition}
\newtheorem{defn}[thm]{Definition}

\theoremstyle{myremark}
\newtheorem{rem}[thm]{Remark}
\numberwithin{equation}{section}

\def\C{\mathbb C}

\def\E{\mathbb E}

\def\N{\mathbb N}

\def\R{\mathbb R}
\def\p{\mathbb P}

\def\Z{\mathbb Z}

\newcommand{\CB}{\mathcal{B}}

\newcommand{\CE}{\mathcal{E}}
\newcommand{\CF}{\mathcal{F}}

\newcommand{\h}{\mathcal{H}}

\newcommand{\CI}{\mathcal{I}}

\newcommand{\CM}{\mathcal{M}}
\newcommand{\CN}{\mathcal{N}}

\newcommand{\CQ}{\mathcal{Q}}

\newcommand{\eps}{\varepsilon}

\newcommand{\abs}[1]{\left\vert#1\right\vert}

\newcommand{\norm}[1]{\left\Vert#1\right\Vert}

\newcommand{\half }{\frac{1}{2}}

\newcommand{\nrm}{   \norm{ \cdot }}

\begin{document}
	
	\title[Maximal inequalities for square functions]{Maximal inequalities for square functions and quantitative mean ergodic theorems associated to group metric measure spaces }
	
	\author[P. Bikram]{Panchugopal Bikram}
	\address{School of Mathematical Sciences,
National Institute of Science Education and Research,  Bhubaneswar, An OCC of Homi Bhabha National Institute,  Jatni- 752050, India}
	\email{bikram@niser.ac.in}

	\author[D. Saha]{Diptesh Saha}
	\address{Institute of Mathematics of the Polish Academy of Sciences,  ul. Sniadeckich 8, 
		00–656 Warszawa, Poland}
	\email{dptshs@gmail.com}

	

	
	

\begin{abstract}
In this article, we establish weighted strong and weak  type inequalities for non-commutative square functions that naturally arise in the analysis of differences between ball averages and martingale sequences within the framework of group metric measure spaces. Then we use these maximal inequalities to prove a quantitative mean ergodic theorem. Our study extends classical harmonic analysis techniques to the non-commutative setting, revealing intricate interactions between group structures, operator-valued functions, and associated filtration systems.

\end{abstract}

	\maketitle 

\section{Introduction}

Maximal inequalities play a central role in harmonic analysis and ergodic theory, serving as fundamental tools for studying the behavior of averages, establishing pointwise convergence, and deriving strong and weak type estimates. In the classical setting, square functions and maximal functions are closely related. The square functions offer a refined, quantitative view of oscillation, while maximal functions are crucial for almost everywhere convergence results. Their interplay underlies many classical results, including the Hardy–Littlewood maximal inequality, Stein's square function estimates, and Bourgain's work in pointwise ergodic theory.

Over the past two decades, significant progress has been made in extending classical harmonic analysis techniques to the non-commutative setting, where functions take values in von Neumann algebras. In this framework, studying maximal inequalities for square functions in non-commutative $L^p$-spaces plays a central role, particularly in establishing non-commutative quantitative mean ergodic theorems. In this article, we focus on weighted analogues of certain square function inequalities in the non-commutative setting, motivated by problems arising in ergodic theory.

Let \( G \) be a locally compact group equipped with a left-invariant Haar measure \( \mu \) and the Borel $\sigma$-algebra $\CF$. Suppose further that \( G \) is a complete metric space with respect to a metric \( d \), such that the triple \( (G, d, \mu) \) satisfies the annular decay condition \cref{eq: decay property} and the doubling condition \cref{eq: doubl cond}. We refer to such a structure $(G,d, \mu)$ as a group metric measure space. For metric spaces satisfying a geometric doubling condition, Hyt\"onen and Kairema \cite{TuomasHytonen2012} constructed dyadic cube systems in \( G \) that are analogous to the standard dyadic partitions of \( \R^d \) for some \( d \in \N \). This construction provides a natural filtration of \(\sigma\)-subalgebras $(\CF_n)$ of $\CF$ on the measure space \( (G, \mu) \), which in turn allows one to define a net of conditional expectation operators $(\E_n)$ from \( L^1(G, \CF, \mu) \) to $ L^1(G, \CF_n, \mu)$.

Let \( \CM \) be a semifinite von Neumann algebra equipped with a faithful, normal, semifinite (f.n.s.) trace \( \tau \). For a group metric measure space \( (G, d, \mu) \), we consider the von Neumann algebra \( \CN := L^\infty(G, \mu) \otimes \CM \), equipped with the f.n.s.\ trace \( \varphi := \int_G \otimes \tau \). Furthermore, for each $n \in \N$, there exists a conditional expectation $\CE_n:= \E_n \otimes I$ form $\CN$ to $L^\infty(G, \CF_n, \mu) \otimes \CM$, preserving the trace $\varphi$.

Given any locally integrable function \( f : G \to L^1(\CM, \tau) \), we define the averaging operator by
\begin{equation} \label{avg op}
A_{r} f(x) := \frac{1}{\mu(B(x, r))} \int_{B(x, r)} f(y) \, d\mu(y), \quad x \in G,
\end{equation}
where, $r>0$ and the set \( B(x, r) \) denotes the ball of radius \( r \) centered at the point \( x \).

If we replace \( \CM \) by \( \C \), the field of complex numbers, then the associated classical square function is defined by
\[
L f(x) := \left( \sum_k \left| \left( A_{\delta^k} - \mathbb{E}_k \right) f(x) \right|^2 \right)^{1/2}, \quad x \in G.
\]

In \cite{jones2003oscillation}, the authors considered the special case where \( G = \Z^d \) for some \( d \in \N \), \( \delta = 2 \), and \( \CM = \C \). In this setting, they proved that the operator \( L \) is bounded from \( \ell^p \) to \( \ell^p \) for all \( 1 < p < \infty \), and from \( \ell^1 \) to weak-\( \ell^1 \).

To obtain a similar result in the non-commutative setting, one first needs to introduce suitable function spaces, namely the vector-valued non-commutative spaces \( L^p(\CN; \ell_2^{rc}) \) and \( L^{1,\infty}(\CN; \ell_2^{rc}) \) (cf.~\cref{subsec: vector valued Lp}). In this framework, the non-commutative square function is associated with the sequence
\begin{equation} \label{nc1.2}
T_k f(x) := (A_{\delta^k} - \mathbb{E}_k) f(x), \quad x \in G,
\end{equation}
where \( f \in L^1(\CN, \varphi) \). In \cite{hong2021noncommutative}, the authors studied this problem in the setting where \( G = \mathbb{R}^d \), \( \delta = \frac{1}{2} \), and \( \CM \) is a semifinite von Neumann algebra. They established the following result:

\begin{thm}[\cite{hong2021noncommutative}] \label{thm: hong thm}
Let \( 1 \leq p < \infty \). Then there exists a constant \( C_p > 0 \), depending only on \( p \), such that
\begin{equation*}
\begin{cases}
\left\| (T_k f) \right\|_{L^{1, \infty}(\CN; \ell_2^{rc})} \leq C_1 \|f\|_1, & \forall f \in L^1(\CN), \\\\
\left\| (T_k f) \right\|_{L^p(\CN; \ell_2^{rc})} \leq C_p \|f\|_p, & \forall f \in L^p(\CN), \quad \text{for } 1 < p < \infty.
\end{cases}
\end{equation*}
\end{thm}

The study of weighted inequalities has long been a central theme in classical harmonic analysis, with a rich history of extending unweighted estimates to settings that incorporate Muckenhoupt-type weights. Naturally, this prompts the question of whether similar weighted results can be developed in the non-commutative framework. The classical theory was pioneered by Muckenhoupt in his foundational work \cite{muckenhoupt1972weighted}, and subsequently advanced by many others, including \cite{fefferfeffer12}. A landmark development in this area was the resolution of the so-called \( A_2 \)-conjecture, which sought to determine whether certain operator bounds exhibit linear dependence on the \( A_2 \)-characteristic of the weight. This was ultimately established by Hyt\"onen in his breakthrough result \cite{Hytonen12ann}.

Parallel developments in the theory of operator algebras and quantum probability have driven rapid progress in non-commutative analysis. Notable contributions in this direction include results on martingales, maximal inequalities, and ergodic theorems within non-commutative \( L^p \)-spaces, as seen in the works \cite{Junge2007, Mei07, Homglaixu23, Hongraywang23, Hongliaowang21, hong2022noncommutative}. Despite these advancements, the incorporation of weights into the non-commutative setting has seen comparatively limited exploration. A recent notable step was taken in \cite{galkazka2022sharp}, where the authors proved weighted non-commutative Doob's maximal inequalities. In a recent work, \cite{saha-rayweighted}, using the Calder\' on-Zygmund decomposition developed by \cite{cadilhac2022spectral} and \cite{cadilhac2022noncommutative}, and weighted Doob's inequality, studied in \cite{galkazka2022sharp}, the authors proved a weighted strong and weak type bounds for non-commutative square functions extending \Cref{thm: hong thm}. 

Returning to the classical setting, the boundedness of square functions associated with group metric measure spaces \( (G, d, \mu) \) has been investigated in \cite{hong2021quantitative}. In this work, we extend those results by establishing the following weighted analogue of \Cref{thm: hong thm}. 

\begin{thm}\label{thm: main thm-intro}
Let \( 1 \leq p < \infty \), and let \( w \) be an \( A_1 \)-weight (cf. \cref{subsec: muckenhoupt's weight}). Then there exists a constant \( C_p(w) > 0 \), depending only on \( p \) and the weight \( w \), such that the following bounds hold:
\begin{equation*}
\begin{cases}
\left\| (T_k f) \right\|_{L^{1, \infty}(\mathcal{N}_w; \ell_2^{rc})} \leq C_1(w) \|f\|_{1, w}, & \text{for all } f \in L^1(\mathcal{N}_w), \\\\
\left\| (T_k f) \right\|_{L^p(\mathcal{N}_w; \ell_2^{rc})} \leq C_p(w) \|f\|_{p, w}, & \text{for all } f \in L^p(\mathcal{N}_w), \quad \text{where } 1 < p < \infty.
\end{cases}
\end{equation*}
\end{thm}

On the other hand, from the von Neumann's ergodic theorem, it is well-known that for an unitary operator $U$ defined on a Hilbert space the ergodic averages
\begin{align*}
    M_n(U)(\xi):= \frac{1}{n} \sum_{k=0}^{n-1} U^k \xi 
\end{align*}
converges in Hilbert space norm to the fixed point space under the unitary $U$. Furthermore, this result was extended for any contraction $T$ on classical $L^p$ spaces $L^p(X, \mu)$ \cite{riesz1941another}. Later on Jones, Ostrovski, and Rosenblatt \cite{jones1996square} proved that for any contraction $T$ on $L^2(X, \mu)$ any finite sequence of positive numbers $n_1<n_2< \ldots <n_r$,
\begin{align*}
    \sum_{i=1}^r \norm{M_{n_i}(T) f - M_{n_{i-1}}(T) f}_{L^2}^2 \lesssim \norm{f}_{L^2}^2~ \text{for any } f \in L^2(X, \mu).
\end{align*}

This result is considered as the quantitative and finer version of mean ergodic theorem. We refer to \cite{bourgain1989pointwise}, \cite{jones1998oscillation}, \cite{jones2003oscillation} and references there in for more information. Very recently, similar result is extended for the group metric measure space setting \cite{HONG_LIU_2026}.  With the recent developments in non-commutative ergodic theory (cf. \cite{junge2007noncommutative}, \cite{hong2021noncommutative}, \cite{cadilhac2022noncommutative}), it is natural to ask how much of these classical quantitative ergodic theorems can be extended to the non-commutative setup. Recently in \cite{hong2024quantitative}, the authors proved quantitative mean ergodic theorem for power bounded operators on non-commutative $L^p$-spaces for $1<p< \infty$. In this article, we extend the main results of \cite{hong2024quantitative} to prove the following theorem.

\begin{thm}\label{quant erg thm}
    Let $1<p< \infty$ and $d$ is a left invariant metric on $G$. Now assume that $\alpha=(\alpha_g)$ is an action of $(G,d, \mu)$ on non-commutative space $L^p(\CM, \tau)$  by invertible power bounded operators. Suppose 
    \begin{equation*}
    M_r(x):= \frac{1}{\mu(B_r)} \int_{B_r} \alpha_g(x) d \mu(g); ~ x \in L^p(\CM, \tau), ~ r> 0.
\end{equation*}
Then there exists a constant $C_p>0$ such that for all $x \in L^p(\CM, \tau)$
    \begin{equation*}
        \sup  \norm{((M_{ r_i}-M_{ r_{i+1}})(x))_{i \in \N}}_{L^p(\CM; \ell_2^{rc})}  \leq C_p \norm{x}_{L^p(\CM, \tau)},
    \end{equation*}
    where  the supremum is considered over all increasing sequence $(r_i)$ of strictly positive numbers.
\end{thm}

We conclude this section by outlining the structure of the paper. In \Cref{sec: prelim}, we recall the necessary background on metric measure spaces, non-commutative \( L^p \)-spaces, and Muckenhoupt weights. We also present the non-commutative Calderón–Zygmund decomposition, which serves as a key tool in our analysis. One of the main results, \Cref{thm: main thm-intro}, is proved in \Cref{sec: main thm}. In \Cref{sec: Quant erg thm}, we further assume that the group $G$ has an invariant metric, which satisfies a doubling condition, and prove an associated Quantitative ergodic theorem.

\section{Preliminaries}\label{sec: prelim}

Let \((X, d)\) be a metric space. For \(x \in X\) and \(r > 0\), we define the ball centered at \(x\) with radius \(r\) by
\[
B(x, r) := \{ y \in X : d(y, x) \leq r \}.
\]
Assume further that \(X\) is equipped with a Borel measure \(\mu\) such that
\[
0 < \mu(B(x, r)) < \infty \quad \text{for all } x \in X \text{ and } r > 0.
\]
The measure \(\mu\) is said to satisfy the doubling condition if there exists a constant \(C_d > 0\) such that
\begin{equation}\label{eq: doub cond-defn}
\mu(B(s, 2r)) \leq C_d \, \mu(B(s, r)) \quad \text{for all } s \in X \text{ and } r > 0.
\end{equation}

Throughout this note, we will use the following notation: for non-negative quantities \(a\) and \(b\), we write \(a \lesssim b\) to mean that there exists a constant \(C > 0\) such that \(a \leq Cb\). We write \(a \approx b\) if both \(a \lesssim b\) and \(b \lesssim a\).

\subsection{Muckenhoupt's Weights}\label{subsec: muckenhoupt's weight}

We begin with some basic notions related to weights on a metric space. Let \((X, d)\) be a complete metric space equipped with a Borel measure \(\mu\) satisfying the doubling condition~\eqref{eq: doub cond-defn}.

A function \(f: X \to \mathbb{R}\) is said to be locally integrable if it is integrable over every compact subset of \(X\); we denote this by \(f \in L^1_{\text{loc}}(X)\). A Borel measure \(\nu\) on \(X\) is said to be absolutely continuous with respect to \(\mu\) if there exists a non-negative function \(w \in L^1_{\text{loc}}(X)\) such that
\[
d\nu = w \, d\mu.
\]
In this case, \(\nu\) is referred to as a weighted measure with respect to \(\mu\), and the function \(w\) is called a weight. For any measurable subset \(F \subseteq X\) and a weight \(w\), we define
\[
w(F) := \int_F w \, d\mu.
\]

\begin{defn}
 Let $w$ be a weight on $X$. Let $1<p < \infty$. Then, we say that $w$ satisfies Mukenhoupt's $A_p$ condition if 
    \[
    [w]_{A_p}:= \sup_{x \in X, r>0} \left( \frac{1}{\mu(B(x,r))} \int_{B(x,r)} w d\mu \right) \left( \frac{1}{\mu(B(x,r))} \int_{B(x,r)} w^{1/(1-p)} d \mu \right)^{p-1}< \infty,
    \]
    and in this case, we write $w \in A_p(X)$. If $p=1$, we define $A_1(X)$ to be the class of weights $w$ for which there exists  $C>0$ such that 
    \[
    \frac{1}{\mu(B(x,r))} \int_{B(x,r)} w d\mu \leq C \inf_{t \in B(x,r)} w(t)
    \]
    for all $x \in X$ and $r>0$. The smallest such $C>0$ satisfying the above equation is denoted by $[w]_{A_1}$.
\end{defn}
In the following, we record some properties of weights satisfying Muckenhoupt's condition. For proofs we refer to the article \cite{kurki2022extension}.

\begin{thm}\label{thm: Muckenhoupt prop}
Let $(X,d, \mu)$ be a metric measure space satifying the doubling condition.
\begin{enumerate}
    \item  Let $1 \leq p \leq q < \infty$. Then $A_p(X) \subseteq A_q(X)$.
    \item Let $p > 1$, and $v$ a weight on $X$ such that $v^r \in A_p(X)$ for some $r>1$. Then there exists $v_1, v_2 \in A_1(X)$ such that $v= v_1 v_2^{1-p}$.
    \item Let $1 \leq p < \infty$ and $w \in A_p(X)$.
\begin{enumerate}
    \item (\textbf{Reverse Holder's Inequality}) There exists $\delta>0$ and $0<C< \infty$ such that 
    \[
    \left( \frac{1}{\mu(B(x,r))} \int_{B(x,r)} w^{1+\delta} d\mu \right)^{1/(1+ \delta)} \leq C \frac{1}{\mu(B(x,r))} \int_{B(x,r)} w d\mu
    \]
    for all $x \in X$ and $r>0$.
    \item There exist $\epsilon>0$ such that $w^{1+\epsilon} \in A_p(X)$.
    \item Let $B$ be an open ball of $X$, and $E \subseteq B$ with $\mu(E)>0$, then
    \[
    \int_B w d\mu \leq [w]_{A_p} \left(\frac{\mu(B)}{\mu(E)}\right)^p \int_E w d\mu.
    \]
    \item There exists constants $C=C(w)>0$ and $0< \upsilon=\upsilon(w) < 1$ (depending only on the doubling constant and the weight $w$) such that for all open balls $B$ and all measurable subsets $E \subseteq B$ we have
    \[
    \frac{w(E)}{w(B)} \leq C \left( \frac{\mu(E)}{\mu(B)} \right)^{\upsilon}.
    \]
\end{enumerate} 
\end{enumerate}
\end{thm}

\subsection{Weighted $L^p$-spaces}
Let \(\CM\) be a semifinite von Neumann algebra equipped with a faithful, normal, semifinite (f.n.s.) trace \(\tau\). A (possibly unbounded) operator \(x\), which is closed and densely defined on a Hilbert space \(\h\), is said to be affiliated with \(\CM\) if it commutes with all unitaries \(u'\) in the commutant \(\CM'\). The operator \(x\) is called \(\tau\)-measurable if for every \(\delta > 0\), there exists a projection \(p \in \mathcal{P}(\CM)\) such that \(p\mathcal{H} \subseteq \mathcal{D}(x)\) (the domain of \(x\)) and \(\tau(1 - p) < \delta\). 

The set of all \(\tau\)-measurable operators affiliated with \(\CM\) is denoted by \(L^0(\CM, \tau)\). It is well known that \(\tau\) extends to the positive cone \(L^0(\CM, \tau)_+\), and for \(x \in L^0(\CM, \tau)\), one defines the \(L^p\)-quasi-norm by
\[
\|x\|_p := \left( \tau(|x|^p) \right)^{1/p}, \quad 0 < p < \infty,
\]
where \(|x| := (x^*x)^{1/2}\). The non-commutative \(L^p\)-space associated with \((\CM, \tau)\) is then given by
\[
L^p(\CM, \tau) := \{ x \in L^0(\CM, \tau) : \|x\|_p < \infty \}, \quad \text{for } 0 < p < \infty,
\]
and we define \(L^\infty(\CM) := \CM\). For \(1 \leq p \leq \infty\), the space \(L^p(\CM, \tau)\) is a Banach space with respect to the norm \(\|\cdot\|_p\).

On the other hand, the non-commutative weak \(L^p\)-space associated with \((\CM, \tau)\) is a subspace of \(L^0(\CM, \tau)\) consisting of those \(x\) for which the following quasi-norm is finite:
\[
\|x\|_{p, \infty} := \sup_{\lambda > 0} \lambda \cdot \tau\left( \chi_{(\lambda, \infty)}(|x|) \right)^{1/p}.
\]
We denote this space by \(L^{p,\infty}(\CM)\), which is a quasi-normed space. In particular, for \(p = 1\), if \(x_1, x_2 \in L^{1,\infty}(\CM)\) and \(\lambda > 0\), the following inequality holds:
\begin{equation}\label{eq: triangle ineq}
\tau\left( \chi_{(\lambda, \infty)}(|x_1 + x_2|) \right) 
\leq \tau\left( \chi_{(\lambda/2, \infty)}(|x_1|) \right) 
+ \tau\left( \chi_{(\lambda/2, \infty)}(|x_2|) \right).
\end{equation}

In this article, we consider weights of the form \( w \otimes I_{\CM} \), where \( w \) is a classical weight as defined above. Observe that \( w \otimes I_{\CM} \) commutes with all elements of the von Neumann algebra \( \CN := L^\infty(X, \mu) \otimes \CM \). We say that \( w \otimes I_{\CM} \) satisfies the Muckenhoupt \( A_p \) condition if and only if \( w \) satisfies the classical \( A_p \) condition. Accordingly, we define
\[
[w \otimes I_{\CM}]_{A_p} := [w]_{A_p}.
\]
From now on, we will refer to such weights simply by \( w \), without ambiguity.

For a weight \( w \) of this form, the weighted non-commutative \( L^p \)-space (for \(1 \leq p < \infty\)) is defined as
\[
L^p_w(\CN) := \left\{ x \in L^0(\CN, \varphi) : x w^{1/p} \in L^p(\CN, \varphi) \right\},
\]
and is equipped with the norm
\[
\norm{x}_{L^p_w(\CN)} := \left( \varphi(\abs{x}^p w) \right)^{1/p}.
\]

Equivalently, defining the weighted f.n.s.\ trace \( \varphi_w(\cdot) := \varphi(\cdot\, w) \) on the von Neumann algebra \( \CN \), it is evident that \( L^p_w(\CN) \) coincides with the non-commutative \( L^p \)-space associated to the pair \( (\CN, \varphi_w) \). In the sequel, we may occasionally use the notation \( \CN_w \) to emphasize that \( \CN \) is considered with the trace \( \varphi_w \). The associated Banach space norm will be denoted by \( \nrm{}_{p,w} \).

\subsection{Vector-valued non-commutative $L^p$-spaces}\label{subsec: vector valued Lp}
Let \((X, \mu)\) be a measure space and \((\CM, \tau)\) a semifinite von Neumann algebra. We recall the definition of the column space from \cite{Mei07, pisier2003non}:
\[
L^p(\CM; L^2_c(X)) := \left\{ f : X \to L^p(\CM) ~\middle|~ f \text{ is measurable and } \|f\|_{L^p(\CM; L^2_c(X))} < \infty \right\},
\]
where
\[
\|f\|_{L^p(\CM; L^2_c(X))} := \left\| \left( \int_X f^*(x) f(x) \, d\mu(x) \right)^{1/2} \right\|_p,
\]
which defines a norm on \(L^p(\CM; L^2_c(X))\) for \(1 \leq p < \infty\). Similarly, the associated row space is defined with the norm
\[
\|f\|_{L^p(\CM; L^2_r(X))} := \left\| \left( \int_X f(x) f^*(x) \, d\mu(x) \right)^{1/2} \right\|_p.
\]

We now recall the following Hölder-type inequality.

\begin{prop}[\cite{Mei07}, Proposition~1.1]\label{prop:Holder-type-ineq}
Let \(1 \leq p, q \leq \infty\) and \(\frac{1}{r} = \frac{1}{p} + \frac{1}{q}\). If \(f \in L^p(\CM; L^2_c(X))\) and \(g \in L^q(\CM; L^2_c(X))\), then
\[
\left\| \int_X f^*(x) g(x) \, d\mu(x) \right\|_r \leq \|f\|_{L^p(\CM; L^2_c(X))} \cdot \|g\|_{L^q(\CM; L^2_c(X))}.
\]
\end{prop}

\noindent
Let \(1 \leq p \leq \infty\), let \(w\) be a weight, and consider the von Neumann algebra \(\CN_w\). Observe that when \(X = \mathbb{N}\), equipped with the counting measure, the norms on the column and row spaces associated with \(\CN_w\) take the following discrete forms:
\[
\norm{ (f_k) }_{ L^{ p } ( \CN_w; \ell_2^c) } = \norm{\left( \sum_{ k} \abs{ f_k}^2 \right)^\half }_{ p,w } \text{ and }  \norm{ (f_k) }_{ L^{ p } ( \CN_w; \ell_2^r) } = \norm{\left( \sum_{ k} \abs{ f_k^*}^2 \right)^\half }_{ p,w }. 
\]
The associated row-column space $L^p(\CM; \ell_2^{rc})$ is defined as follows.
\begin{enumerate}
\item If \( 2 \leq p \leq \infty \),
  \[
    L^p(\CN_w; \ell_2^{rc}) := L^p(\CN_w; \ell_2^c) \cap L^p(\CN_w; \ell_2^r)
  \]
  equipped with the norm:
  \[
    \| (f_k) \|_{L^p(\CN_w; \ell_2^{rc})} = \max \left\{ \| (f_k) \|_{L^p(\CN_w; \ell_2^c)}, \| (f_k) \|_{L^p(\CN_w; \ell_2^r)} \right\}.
  \]

  \item If \( 1 \leq p < 2 \),
  \[
    L^p(\CN_w; \ell_2^{rc}) = L^p(\CN_w; \ell_2^c) + L^p(\CN_w; \ell_2^r)
  \]
  equipped with the norm:
  \[
    \| (f_k) \|_{L^p(\CN_w; \ell_2^{rc})} = \inf \left\{ \| (g_k) \|_{L^p(\CN_w; \ell_2^c)} + \| (h_k) \|_{L^p(\CN_w; \ell_2^r)} \right\},
  \]
  where the infimum runs over all decompositions \( f_k = g_k + h_k \) with \( g_k \) and \( h_k \) in \( L^p(\CN_w) \).
\end{enumerate}
In a similar manner, one can define the weak-type spaces $L^{1, \infty}(\CN_w; \ell_2^c)$ and $L^{1, \infty}(\CN_w; \ell_2^r)$. Correspondingly, the associated row-column space $L^{1,\infty}(\CN_w; \ell_2^{rc})$ is also defined, equipped with the quasi-norm
\[
\norm{ (f_k) }_{ L^{ 1, \infty } ( \CN_w; \ell_2^{rc}) } = \inf_{ f_k = g_k + h_k}
\{  \norm{ (g_k) }_{ L^{ 1, \infty } ( \CN_w; \ell_2^c) } +  \norm{ (h_k) }_{ L^{ 1, \infty } ( \CN_w; \ell_2^r) } \}.
\]

Let us now recall the following non-commutative Khintchine inequality, which will play a crucial role in our analysis. For the proof, we refer to \cite{lust1991non} and \cite[Corollary 3.2]{cadilhac2019noncommutative}.

\begin{thm}\label{thm: Khintchine}
    Let $w \in A_1$ and $(\eps_k)$ be a Rademacher sequence on a probability space $(\Omega, \p)$. Then,
    \begin{enumerate}
        \item $\norm{ \sum_{ k } \varepsilon_k f_k  }_{ L^{ p } ( L^\infty( \Omega) \otimes \CN_w) } \approx  \norm{ (f_k) }_{ L^{ p} ( \CN_w; \ell_2^{cr}) }.$
        \item $\norm{ \sum_{ k } \varepsilon_k f_k  }_{ L^{ 1, \infty } ( L^\infty( \Omega) \otimes \CN_w) } \approx  \norm{ (f_k) }_{ L^{ 1, \infty } ( \CN_w; \ell_2^{cr}) }$.
    \end{enumerate}

\end{thm}

\subsection{Non-commutative martingales}
Let $\CM$ be a semifinite von Neumann algebra with an f.n.s trace $\tau$ and let $(\CM_n)_{n \geq 0}$ be a decreasing sequence of von Neumann subalgebras of $\CM$ such that $\CM_0= \bigcup_{n \geq 0} \CM_n$ is $w^*$-dense in $\CM$. Then, for each $n$, there exists a $\tau$-preserving normal conditional expectation $\CE_n: \CM \to \CM_n$. 

A sequence $(x_n)_{n \geq 0}$ in $L^1(\CM, \tau)$ is called a non-commutative  martingale if $\CE_n(x_{n+1})= x_{n+1}$ for all $n\geq 0$. The associated martingale difference sequence $(dx_n)_{n \geq 0}$ is defined by $dx_n:= x_n-x_{n+1}$ for all $n \geq 0$.

We now turn our attention to the specific class of martingales that are the focus of this article. From this point onward, unless otherwise stated, we will replace the general metric space \(X\) with a locally compact group \(G\), and \(\mu\) will denote the left-invariant Haar measure on \(G\). The identity element of the group will be denoted by \(e\). In addition, we assume that \(G\) is equipped with a complete metric \(d\), making it also a metric space. For each \(r > 0\), we denote by \(B_r\) the ball centered at the identity \(e\), i.e.,
\[
B_r := \{ g \in G : d(g, e) \leq r \}.
\]

Let $r_0>0$ and $0< \epsilon \leq 1$ be fixed. We say that the triple $(G,d, \mu)$ satisfies the $(\epsilon,r_0)$-annular decay property if there exists a constant $K>0$ such that for all $x\in G$, $r\in (r_0,\infty)$, and $s\in (0,r]$, the following inequality holds:
		\begin{equation}\label{eq: decay property}
			\mu(B(x,r+s))-\mu(B(x,r))\le K\bigg(\frac{s}{r}\bigg)^{\epsilon}\mu(B(x,r)).
	\end{equation}
It follows from \cref{eq: decay property} that for all $x \in G$ and $r_0< r \leq R< \infty$, we have the estimate:
    \begin{equation}\label{eq: doubl cond}
        \frac{\mu(B(x,R))}{\mu(B(x,r))}\leq (K+1)\Big(\frac{R}{r}\Big)^{\epsilon}.
    \end{equation}
That is, the balls of radius greater than $r_0$ satisfies doubling condition with respect to the measure $\mu$. We now recall the following useful lemma from \cite[Lemma 2.9]{hong2021quantitative}, which will be useful in the sequel.
 
\begin{lem}\label{lem: annular decay}
    Let $r \geq 2r_0$. Then there exists a constant $K_\epsilon>0$ such that for every $0< s \leq r$ and for all $x \in G$, the following inequality holds:
    \[
    \mu(B(x, r+s))- \mu(B(x, r-s)) \leq K_\epsilon \left( \frac{s}{r} \right)^\epsilon \mu(B(x, r)).
    \]
\end{lem}

Let $D_0>0$ and $r_1>r_0$ be fixed. We say that the metric space $(G,d)$ satisfies the $(D_0,r_1)$-geometrically doubling property if for every $r\in(0,r_1]$ and every ball $B(h,r)$, there are at most $D_0$ balls $B(h_i,r/2)$ such that
	\begin{equation}\label{eq: geo-doubling}
		B(h,r)\subseteq \bigcup_{1\le i\le{D_0}}B(h_i,r/2).
	\end{equation}

In the sequel, we will refer to any triple $(G,d, \mu)$ satisfying the conditions \ref{eq: decay property} and \ref{eq: geo-doubling} as a group metric measure space.

    \begin{prop}\cite{TuomasHytonen2012}\label{prop: hytonen lem}
        Let $(G,d,\mu)$ be a group metric measure space. Fix constants $0<c_0<C_0<\infty$ and $\delta>1$ such that
        \[
        18C_0\delta^{-1}\leq c_0.
        \]
        Let $I_k$ \emph{($k\in \N$)} be an index set and $\{z_\alpha^k\in G:\alpha\in I_k,k\in \N\}$ be a collection of points with the properties that
\begin{equation}\label{distance}
  d(z_\alpha^k,z_\beta^k)\ge c_0\delta^k~(\alpha\neq\beta),~\min_{\alpha}d(x,z_\alpha^k)<C_0\delta^k,~\forall~x\in G,~k\in\mathbb{Z}.
\end{equation}
Then there exist a family of sets $\big\{Q_\alpha^k\big\}_{\alpha\in I_k}$ associating with $\{z_\alpha^k\}_{\alpha\in I_k}$, and constants $a_0:=c_0/3$ and $C_1:=2C_0$ such that
\begin{enumerate}
    \item $\forall~k\in\Z$,~$\cup_{\alpha\in I_k} Q_\alpha^k=G$.
    \item If $k\leq l$ then either $Q_\alpha^k\subset Q_\beta^l$ or $Q_\alpha^k\cap Q_\beta^l=\emptyset$.
    \item For each $(k,\alpha)$ and each $k<n$ there is a unique $\beta$ such that $Q_\alpha^k\subset Q_\beta^n$, and for $n=k+1$, we call such $Q_\beta^{k+1}$ the parent of $Q_\alpha^{k}$.
    \item $B(z_\alpha^k, a_0\delta^k)\subseteq Q_\alpha^k\subseteq B(z_\alpha^k, C_1\delta^k)$.
\end{enumerate}
    \end{prop}

Let $k \in \N$ and  $\CB_k$ be the $\sigma$-algebra generated by the 'dyadic cubes' $ \CQ_k=:\{Q^k_\alpha: \alpha \in I_k\}$.  Now, for every $k \in \N$, we define the von Neumann subalgebra
	\begin{align*}
		\CN_k:= L^\infty(G, \CB_k, \mu) \otimes \CM
	\end{align*}
    of $\CN:= L^\infty(G, \CB_k, \mu) \otimes \CM$. Observe that for $k\leq n$, $\CN_n$ is a subalgebra of $\CN_k$. Furthermore, the von Neumann algebra $\CN$ is equipped with the f.n.s trace $\varphi= \int_G \otimes \tau$.
    
    Now since $\varphi|_{\CN_k}$ is again semi-finite, there exists a normal conditional expectation $\CE_k$ which is of the form $\CE_k:= \E_k(\cdot | \CB_k) \otimes I_{\CM}$, where $\E_k(\cdot | \CB_k)$ is the conditional expectation from $L^\infty(G, \CB, \mu)$ onto $L^\infty(G, \CB_k, \mu)$. Therefore, for $k \in \Z$ and $f \in L^1(\CN, \varphi)$ we have
	\begin{align*}
		f_k:= \CE_k(f)=\sum_{\alpha \in I_k} f_{Q^k_\alpha} \chi_{Q^k_\alpha}, 
	\end{align*}
	where $f_{Q^k_\alpha}:= \frac{1}{\mu(Q^k_\alpha)} \int_{Q^k_\alpha} f d \mu$. Moreover, observe that $(\CN_k)_{k \in \Z}$ is a decreasing family of subalgebras of $\CN$. Hence, for every $f \in L^1(\CN, \varphi)$, the sequence $(\CE_k(f))_{k \in \Z}$ forms a non-commutative martingale, that is 
	\begin{align*}
		\CE_j(\CE_k(f))= \CE_{\max\{j,k\}}(f).
	\end{align*}

Let \( (G, d, \mu) \) be a group metric measure space. Let \( f: G \to L^1(\mathcal{N}, \varphi_w) \) be a function for which the averaging operator is defined as
\[
A_r(f)(h) := \frac{1}{\mu(B(h, r))} \int_{B(h, r)} f(g) \, d\mu(g), \quad r > 0,~ h \in G.
\]

Let \(\delta > 1\) be the constant determined in \Cref{prop: hytonen lem}, and let \(n_{r_0}\) be the unique integer such that
\[
\delta^{n_{r_0}} < r_0 \leq \delta^{n_{r_0} + 1}.
\]
Now, for every \(n > n_{r_0}\), we define the operator
\[
T_nf(x):= (A_{\delta^n}- \CE_n)f (x),~ f \in L^1_w(\CN, \varphi).
\]

\begin{rem}

We note that given a cube $ Q_\alpha^k $ with $ k > n_{ r_0} $, we have the following 
$$ B(z_\alpha^k, a_0\delta^k)\subseteq Q_\alpha^k\subseteq B(z_\alpha^k, C_1\delta^k)
$$
For minutes, for simplicity of notation, suppose $ Q= Q_\alpha^k$, $ B_0 = B(z_\alpha^k, a_0\delta^k$ and $ B_1 = B(z_\alpha^k, C_1\delta^k)$.
It follows that 
$$ B_0 \subseteq Q \subseteq B_1.$$
Then we note that 
\begin{align*}
    \frac{\frac{1}{\mu(Q) } \int_Q w d\mu}{ \text{essinf}_{Q} (w)}
    &\leq \frac{ \frac{1}{\mu(B_0)} \int_{ B_1} w d\mu }{ \text{essinf}_{B_1} (w)} \\
&=\frac{ \mu(B_1)}{\mu(B_0)}  \frac{\frac{1}{\mu( B_1)} \int_{ B_1 } w d\mu}{ \text{essinf}_{B_1} (w)} \\
&\lesssim [w]_{ A_1}
\end{align*}

\end{rem}

We now state the main result of this section.

\begin{thm}\label{thm: main thm}
Let \(1 \leq p < \infty\), and let \(w\) be an \(A_1\)-weight. Then there exists a constant \(C_p(w)\), depending only on \(p\) and \(w\), such that
\begin{equation}
\begin{cases}
\left\| (T_k f) \right\|_{L^{1, \infty}(\mathcal{N}_w; \ell_2^{rc})} \leq C_1(w) \|f\|_{1,w}, \quad \forall f \in L^1(\mathcal{N}_w), \\\\
\left\| (T_k f) \right\|_{L^p(\mathcal{N}_w; \ell_2^{rc})} \leq C_p(w) \|f\|_{p,w}, \quad \forall f \in L^p(\mathcal{N}_w),~ \text{when } 1 < p < \infty.
\end{cases}
\end{equation}
\end{thm}

\subsection{Non-commutative Calder\'on-Zygmund decomposition}
In this subsection, we describe the non-commutative Calder\'on-Zygmund decomposition from \cite{cadilhac2022noncommutative}, which is based on the construction by Cuculescu \cite{cuculescu1971martingales}. Let us first consider the following class of operator valued functions:
\begin{align*}
		\CN^c_{w, +}:= \{ f:G \to \CM \cap L^1(\CM): f \geq 0,~ \overrightarrow{\text{supp}} f \text{ is compact } \}.
\end{align*}
Observe that, $\CN^c_{w, +}$ is dense in $L^1(\CN_w)_+$. Here, $\overrightarrow{\text{supp}} f$ denotes the support of $f$ as an operator-valued function defined on $G$. Furthermore, using \Cref{prop: hytonen lem} and \cref{eq: doubl cond}, we note that the following holds for $f \in \CN^c_{w, +}$ and $\alpha \in I_k$:
\[
\frac{1}{\mu(Q^k_\alpha)} \int_{Q^k_\alpha} f d \mu \lesssim \delta^k \norm{f}_\infty \mu(\overrightarrow{\text{supp}} f) 1_\CM \rightarrow 0~ \text{as } k \to \infty.
\]
Therefore, given $f \in \CN^c_{w, +}$ and $\lambda>0$, we can find $N \in \N$ such that $f_k \leq \lambda 1_\CN$ for all $k \geq N$.

Before recalling the  following Calder\'on-Zygmund decomposition, we recall the following 
$$
\delta^{ n_{ r_0} } < r_0 \leq \delta^{n_{r_0} +1 }. 
$$
In this article,  we consider the following cubes $ (Q_\alpha^k)_{ k > n_{ r_0} }$. Further,  for simlicity notation, there is no harm to  assume that $ n_{r_0} = -1$ at least for the Calder\'on-Zygmund decomposition. 
 \begin{thm}\cite{parcet2009pseudo}\label{thm: Cuculescu}
 Let $w \in A_1$, $ f \in \CN^c_{ w, +}$. Consider the associated $L^1$-martingale $ (f_k)_{k \geq 0} $.  Then for a given $\lambda >0$, there exists an increasing  sequence of projection $(q_k)_{ k \geq 0} $ in $  \CN_w$, determined inductively by the relation $q_k = 1_{(0, \lambda]}(q_{k+1} f_k q_{ k+1} )$ such that 
 	\begin{enumerate}
    \item $ q_k \in \CN_k$ and $q_k$ commutes with $q_{k+1} f_k q_{k+1 } $  for all $ k $.
    \item For all $ k , ~ q_k f_k q_k \leq \lambda q_k$.
 		\item   Suppose $ q= q_0$, then the following estimate holds:
        \[
        \varphi_w( 1-q) \leq [w]_{ A_1} \frac{\norm{f}_{1,w}}{\lambda}.
        \]
 	\end{enumerate}
\end{thm}

\begin{proof}
Note that $ f \in \CN^c_{ w, +}$ and suppose $ \CE_k(f) = f_k$, then  there exists $N \in \N$ such that 
$$ f_k \leq \lambda, \text{ for all } k \geq N.$$
\medskip 
Let $ q_n = 1 $ for all $ n \geq N$
and 
$$ q_{ k} = 1_{(0, \lambda]} ( q_{ k+1}  f_k q_{ k+1}) $$
Then  note that $(1)$ and $(2)$ follows easily, for $(3)$, observe the followings;
\begin{align*}
     \lambda \varphi_w( 1-q) &= \varphi_w( 1-q_0)\\
&= \lambda \varphi_w( \sum_{ k=0}^{N-1} ( (q_{k+1} -q_k)) \\
& \leq \sum_{ k=0}^{N-1}  \varphi_w(   (q_{k+1} -q_k) f_k  (q_{k+1} -q_k)) \\
& = \sum_{ k=0}^{N-1}  \varphi(   (q_{k+1} -q_k) \CE_k(f) (q_{k+1} -q_k)  w)\\
& = \sum_{ k=0}^{N-1}  \varphi(   (q_{k+1} -q_k) f (q_{k+1} -q_k) \CE_k(w))\\
& = \sum_{ k=0}^{N-1}  \varphi(   (q_{k+1} -q_k) f (q_{k+1} -q_k) \frac{\CE_k(w)}{w} w)\\
&\lesssim [w]_{ A_1} \sum_{ k=0}^{N-1}  \varphi(   (q_{k+1} -q_k) f (q_{k+1} -q_k)  w) \\
&\lesssim  [w]_{ A_1} \norm{f}_{1,w}
\end{align*}
This completes the proof.
\end{proof}

\noindent

Observe that the projections $(q_k)_{k \geq 0}$ obtained in \Cref{thm: Cuculescu} takes the following form.
\[
q_k= \sum_{\alpha \in I_k} q_{Q^k_\alpha} \chi_{Q^k_\alpha}, ~ k \geq 0,
\]
where $q_{Q^k_\alpha}$ is a projection in $\CM$ satisfying
\[
q_{Q^k_\alpha}= \chi_{(0, \lambda]} \left( q_{Q^{k+1}_\beta} f_{Q^{k+1}_\beta} q_{Q^{k+1}_\beta} \right)~  k \geq 0,
\]
where $Q^{k+1}_\beta$ is the parent of $Q^{k}_\alpha$. Clearly, it follows that 
\begin{align}\label{eq: cucu proj prop}
    q_{Q^k_\alpha} \leq q_{Q^{k+1}_\beta},~ q_{Q^k_\alpha} \text{ commutes with } q_{Q^{k+1}_\beta}, ~ \text{and } q_{Q^k_\alpha} f_{Q^k_\alpha} q_{Q^k_\alpha}\leq \lambda q_{Q^k_\alpha}
\end{align}

Furthermore, we define  $p_k= q_{k+1} - q_k$ for $k \geq 0 $ and observe that $p_k$'s are disjoint projections in $\CN$ and 
	\begin{equation*}
		p_k= \sum_{\alpha \in I_k} p_{Q^k_\alpha} \chi_{Q^k_\alpha},~ \text{and}  \qquad \sum_{k \geq 0} p_k = 1- q_0.
	\end{equation*}

Now we have the following lemma.
 \begin{prop}\label{prop: cucu proj prop}
 It  follows that 
     \[
     p_kf_k p_k   \lesssim \lambda p_k, \text{ for all } k \geq 0.
     \]
   \end{prop}   
 \begin{proof}
Consider $s \in G$ and observe that there exists $\alpha(s) \in I_k$ such that $s \in Q^k_{\alpha(s)}$. Now if $Q^{k+1}_{\beta(s)}$ is the parent of $Q^k_{\alpha(s)}$, then we have
     \begin{align*}
         f_k(s)
         &=  \frac{1}{\mu(Q^k_{\alpha(s)})} \int_{Q^k_{\alpha(s)}} f d\mu 
         \leq  \frac{1}{\mu(Q^{k}_{\alpha(s)})} \int_{Q^{k+1}_{\beta(s)}} f d\mu \\
         &= \frac{\mu(Q^{k+1}_{\beta(s)})}{\mu(Q^k_{\alpha(s)})} \frac{1}{\mu(Q^{k+1}_{\beta(s)})} \int_{Q^{k+1}_{\beta(s)}} f d\mu\\
         &= \frac{\mu(Q^{k+1}_{\beta(s)})}{\mu(Q^k_{\alpha(s)})} f_{k+1}(s)
         \leq \frac{\mu(B(z^{k+1}_{\beta(s)},C_1 \delta^{k+1}))}{\mu(B(z^{k}_{\alpha(s)},a_0 \delta^{k}))} f_{k+1}(s)\\
         &\leq \frac{\mu(B(z^{k}_{\alpha(s)},2C_1 \delta^{k+1}))}{\mu(B(z^{k}_{\alpha(s)},a_0 \delta^{k}))} f_{k+1}(s)\\
         &\leq (K+1)( \frac{2C_1 \delta^{k+1}}{a_0 \delta^k})^\eps f_{k+1} = (K+1)( \frac{2C_1 \delta}{a_0 })^\eps f_{k+1}(s).
     \end{align*}
Therefore, the associated martingale filtration is regular and we conclude
     \[
     p_k f_k p_k= p_k q_{k+1} f_k q_{k+1} p_k \lesssim p_kq_{k+1} f_{k+1} q_{k+1} p_k \lesssim \lambda p_k.
     \]
\end{proof}  

The following weighted weak type $(1,1)$ inequality is crucial for our purpose.

	\begin{thm}\cite[Theorem 4.1]{galkazka2022sharp}\label{thm: galkazka's estimate}
		Let $w \in A_1$ and $f \in L^1_w(\CN)_+$. Consider the associated $L^1$-martingale $ (f_k)_{k \geq 0} $. Then for any $\lambda>0$ we have $q f_k q \leq \lambda$ for all $k \geq 0$ and 
		\begin{align*}
			\lambda \varphi_w(1-q) \leq [w]_{A_1} \norm{f}_{1,w}.
		\end{align*}
	\end{thm}
\noindent Now we have the following Calder\'on-Zygmund decomposition.

\begin{thm}\cite{cadilhac2022spectral}\label{thm: CZ decomp}
    Let $w \in A_1$, $f \in L^1_w(\CN)_+$ and $\lambda>0$. Let $(p_k)_k$ and $(q_k)_k$ be the two sequences obtained in \Cref{thm: Cuculescu}. Then $f$ has a decomposition 
    \[
    f= g +b_d + b_{off}
    \]
    which has the following properties.
    \begin{enumerate}
        \item $g= qfq + \sum_{k \geq 0} p_k f_k p_k$ satisfies $\norm{g}_{1,w} \lesssim \max\{1, [w]_{A_1}\} \norm{f}_{1,w}$ and $\norm{g}_\infty \lesssim \lambda$.
        \item $b_d= \sum_{k \geq 0} p_k(f- f_k) p_k$ satisfies $\CE_k(p_k(f- f_k) p_k)=0$.
        \item $b_{off}= \sum_{k \geq 0} p_k(f- f_k) q_k + q_k(f- f_k) p_k$ satisfies $\CE_k \left( p_k(f- f_k) q_k + q_k(f- f_k) p_k \right)=0$.
    \end{enumerate}
\end{thm}

\begin{proof}
    Indeed, observe that,
        	\begin{align*}
			\norm{g}_{L^1_w(\CN)} 
			= \varphi (g w)
			&= \varphi_w (qfq) + \varphi_w \Big(\sum_{k } p_k f_k p_k \Big)\\
			&= \varphi_w(qf) + \varphi_w \Big( \sum_k p_k \CE_k(f)\Big)\\
			&= \varphi_w(qf) + \varphi_w \Big( \sum_k  \CE_k( p_k f)\Big)\\
			&= \varphi_w(qf) + \varphi \Big( \sum_k  \CE_k( p_k f) w \Big)\\
			&= \varphi_w(qf) + \varphi \Big( \sum_k ( p_k f)  \CE_k(w) \Big)\\
			&= \varphi_w(qf) + \varphi \Big( \sum_k ( p_k f)  \frac{\CE_k(w)}{w} w \Big)\\
			&\leq \varphi_w(qf) + [w]_{A_1} \varphi_w \Big( \sum_k ( p_k f) \Big)\\
			&= \varphi_w(qf) + [w]_{A_1} \varphi_w ((1-q)f)\\
			&\leq \max\{1, [w]_{A_1}\} \norm{f}_{1,w}.
		\end{align*}
    
\end{proof}

\noindent 
Now for $k \geq 0$ and $\alpha \in I_k$, let us define,
\[
\tilde{Q}^k_\alpha:= \{ s \in G: d(s, z^k_\alpha) \leq (3C_1+1) \delta^{k+1} \}.
\]
and a projection $\zeta$ in $\CN$ by
\begin{equation}\label{eq: defn zeta}
\zeta:= \left(\bigvee_{k \geq 0} \bigvee_{\alpha \in I_k} p_{Q^k_\alpha} \chi_{\tilde{Q}^k_\alpha}\right)^\perp.
\end{equation}
\noindent
We obtain the following properties.
\begin{lem}\label{lem: zeta estimate}
    Let $w \in A_1$, $f \in L^1_w(\CN)_+$ and $\lambda>0$. Then the following are true.
    \begin{enumerate}
        \item $\varphi_w(\zeta^\perp) \lesssim [w]_{A_1}^2 \frac{\norm{f}_{1,w}}{\lambda}$.
        \item For all $k \geq 0$ and $\alpha \in I_k$, we have the following cancellation property.
        \[
        s\in \tilde{Q}^k_\alpha \Rightarrow \zeta(s) p_{Q^k_\alpha}= p_{Q^k_\alpha} \zeta(s)=0.
        \]
        \item Let $k\geq 0$, $s,t \in G$ and $Q^k_{\alpha(t)}$ be the unique dyadic set containing $t$, for some $\alpha(t) \in I_k$. Then, it follows that whenever $s \in \tilde{Q}^k_{\alpha(t)}$
        \begin{equation*}
            \zeta(s) b_k(t) \zeta(s)=0.
        \end{equation*}
    \end{enumerate}
\end{lem}

    \begin{proof}
    We note that 
    \begin{align*}
    \varphi_w(\zeta^\perp) 
    \leq \sum_{k \geq 0} \sum_{\alpha \in I_k} \varphi_w(p_{Q^k_\alpha} \chi_{\tilde{Q}^k_\alpha})
    &= \sum_{k \geq 0} \sum_{\alpha \in I_k} \int_{\tilde{Q}^k_\alpha} \tau(p_{Q^k_\alpha}) w d\mu\\
    &\leq \sum_{k \geq 0} \sum_{\alpha \in I_k} [w]_{A_1} \frac{\mu(\tilde{Q}^k_\alpha)}{\mu(Q^k_\alpha)} \int_{Q^k_\alpha} \tau(p_{Q^k_\alpha}) w d\mu, ~ \text{ see  } 3(c) \text{ of } \Cref{thm: Muckenhoupt prop} \\
    &\leq \sum_{k \geq 0} \sum_{\alpha \in I_k} [w]_{A_1} \frac{\mu(\tilde{Q}^k_\alpha)}{\mu(B(z_\alpha^k, a_0\delta^k))} \varphi_w(p_{Q^k_\alpha} \chi_{Q^k_\alpha})\\
    &\leq (K+1) \left( \frac{(3C_1+1) \delta^{k+1} }{a_0 \delta^k} \right)^\epsilon [w]_{A_1} \sum_{k \in \Z} \sum_{\alpha \in I_k} \varphi_w(p_{Q^k_\alpha} \chi_{Q^k_\alpha})~ \text{by \cref{eq: doubl cond}}\\
    &= (K+1) \left( \frac{(3C_1 + 1)\delta }{a_0 } \right)^\epsilon [w]_{A_1} \sum_{k \in \Z} \varphi_w(p_k)\\
    &= (K+1) \left( \frac{(3C_1+1)\delta}{a_0 } \right)^\epsilon [w]_{A_1} \varphi_w(1-q)\\
    &\leq (K+1) \left( \frac{(3C_1+1)\delta}{a_0 } \right)^\epsilon [w]_{A_1}^2 \frac{\norm{f}_{1,w}}{\lambda}, \text{by } \Cref{thm: galkazka's estimate}.
    \end{align*} 
    Now let $k \geq 0$, $\alpha \in I_k$ and $s\in \tilde{Q}^k_\alpha$. Then by construction $p_{Q^k_\alpha} \leq \zeta^\perp(s)$. Hence, the result follows.
\end{proof}

    \section{Proof of \Cref{thm: main thm}}\label{sec: main thm}

    In this section, we establish weak type $(1,1)$ and strong type $(p,p)$ estimates for non-commutative square functions. Let $f \in L^1(\CN_w)$. By decomposing
\[
f = f_1 - f_2 + i(f_3 - f_4),
\]
where each $f_j$ is positive and satisfies $\norm{f_j}_{1,w} \leq \norm{f}_{1,w}$, it suffices to prove the maximal inequalities under the assumption $f \in L^1(\CN_w)_+$. Moreover, since the cone \( \mathcal{N}^c_{w,+} \) is dense in $L^1(\CN_w)_+$, it is further sufficient to prove the desired maximal inequalities for $f \in \CN^c_{w, +}$. To estimate weak type $(1,1)$ and strong type $(p,p)$ bounds in \Cref{thm: main thm} we first consider the operator
	\begin{align}\label{eq: linearization}
		Tf(x):= \sum_{k> n_{r_0}} \epsilon_k T_kf (x), ~ f \in \CN^c_{w, +},
	\end{align}
    where $(\eps_k)$ is a sequence of Rademacher variables on a probability space $(\Omega, \p)$.

\subsection{Weak type $(1,1)$ estimate}

We first make the following observation. Let $f \in \CN^c_{w, +}$ and $\lambda>0$. Then, by \cref{eq: triangle ineq} and \Cref{thm: Cuculescu} we have
	\begin{equation*}
		\tilde{\varphi}_w(\abs{Tf}> \lambda) \leq \tilde{\varphi}_w(\abs{Tg}> \lambda/3) + \tilde{\varphi}_w(\abs{Tb_d}> \lambda/3) + 
		\tilde{\varphi}_w(\abs{Tb_{off}}> \lambda/3),
	\end{equation*}
	where, $\tilde{\varphi}_w:= \int_\Omega \otimes \varphi_w$.
    
Therefore, to establish the weak type \((1,1)\) inequality for the non-commutative square function \(T\), it suffices to prove the following theorem.

\begin{thm}\label{thm: main thm weak estimate}
Let \( h \in \{g, b_d, b_{\mathrm{off}}\} \). Then there exists a constant \( C_1(w) > 0 \), depending only on the weight \( w \), such that
		\begin{equation*}
			\lambda \tilde{\varphi}_w ( \abs{Th}> \lambda) \lesssim C_1(w) \norm{f}_{1,w}.
		\end{equation*}
\end{thm}

To prove this result, we first establish some technical lemmas. For all $n, k \in \N$ and $s \in G$, define the sets
\[
\CI(B(s,\delta^k),n):= \bigcup_{\alpha \in I_n,~\partial(B(s,\delta^k)) \cap Q^n_\alpha \neq \emptyset} Q^n_\alpha \cap B(s,\delta^k),
\]
and
\[
\CI_1(B(s,\delta^k),n):= \bigcup_{\alpha \in I_n,~\partial(B(s,\delta^k)) \cap Q^n_\alpha \neq \emptyset} Q^n_\alpha .
\]
\begin{rem}
We observe that for $k>n$
\begin{equation} \label{eq: annular inclusion}
\mathcal{I}(B(s, \delta^k), n) \subseteq \CI_1(B(s,\delta^k),n) \subseteq B(s, \delta^k + C_1 \delta^n) \setminus B(s, \delta^k - C_1 \delta^n).
\end{equation}

Indeed, suppose \( t \in B(s, \delta^k - C_1 \delta^n) \). Then \( t \notin \mathcal{I}(B(s, \delta^k), n) \) by construction. Conversely, if \( t \in \mathcal{I}(B(s, \delta^k), n) \), then there exists \( \alpha \in I_n \) such that
\[
\partial B(s, \delta^k) \cap Q^n_\alpha \neq \emptyset \quad \text{and} \quad t \in Q^n_\alpha \cap B(s, \delta^k).
\]
Moreover, since \( Q^n_\alpha \subseteq B(z^n_\alpha, C_1 \delta^n) \), it follows that
\[
d(s, t) \geq d(s, z^n_\alpha) - d(t, z^n_\alpha) > \delta^k - C_1 \delta^n.
\]
This contradicts the assumption that \( t \in B(s, \delta^k - C_1 \delta^n) \), thereby proving the inclusion \eqref{eq: annular inclusion}.
\end{rem}

\noindent
Now, for \( y \in G \), consider the annular set
\[
A(y) := B(y, \delta^k + C_1 \delta^n) \setminus B(y, \delta^k - C_1 \delta^n).
\]
We now state the following lemma, which will be useful in the forthcoming analysis.

\begin{lem}\label{lem: norm est prop}
    Suppose  $k>n \geq n_{r_0}$ and $y \in A(x)$, then the following holds. 
\begin{align*}
    \int_G \frac{\chi_{A(y)}(x)}{\mu(B(x,\delta^k))} w(x) d\mu(x) \lesssim C(w)[w]_{A_1} w(y) \delta^{(n-k)\epsilon}.
\end{align*}
\end{lem}

\begin{proof}
    Note that,
    \begin{align*}
        &\int_G \frac{\chi_{A(y)}(x)}{\mu(B(x,\delta^k))} w(x) d\mu(x) \\
        &\lesssim \frac{1}{\mu(B(x,\delta^k))} C(w) \left( \frac{\mu(A(y))}{\mu(B(y, \delta^k + C_1 \delta^n))} \right)^\upsilon w(B(y, \delta^k + C_1 \delta^n))\\
        &\leq \frac{1}{\mu(B(x,\delta^k))} C(w) \left( \frac{\mu(A(y))}{\mu(B(y, \delta^k + C_1 \delta^n))} \right)^{\upsilon-1} \frac{\mu(A(y))}{\mu(B(y, \delta^k + C_1 \delta^n))} w(B(y, \delta^k + C_1 \delta^n)) \\
        &\lesssim C(w)[w]_{A_1} \left( \frac{\mu(A(y))}{\mu(B(y, \delta^k + C_1 \delta^n))} \right)^{\upsilon-1} \frac{\mu(A(y))}{\mu(B(x, \delta^k))} w(y)\\
        &\leq C(w)[w]_{A_1} \left( \frac{\mu(B(y, \delta^k + C_1 \delta^n))}{\mu(A(y))} \right)^{1-\upsilon} \frac{\mu(A(y))}{\mu(B(x, \delta^k))} w(y)\\
        &\leq C(w)[w]_{A_1} \left( \frac{\mu(B(y, \delta^k + C_1 \delta^n))}{\mu(B(y, a_0 \delta^k))} \right)^{1-\upsilon} \frac{\mu(A(y))}{\mu(B(x, \delta^k))} w(y)\\
        & \qquad \qquad \text{(since $B(y, a_0 \delta^k) \subseteq \CI_1(B(y,\delta^k),n) \subseteq A(y)$)}\\
        &\leq C(w)[w]_{A_1} (K+1) \left( \frac{\delta^k + C_1 \delta^n}{a_0 \delta^k} \right)^{\epsilon(1-\upsilon)} w(y) \int_G \frac{\chi_{A(y)}(x)}{\mu(B(x,\delta^k))} d\mu(x)\\
        & \lesssim C(w)[w]_{A_1} w(y) \int_G \frac{\chi_{A(y)}(x)}{\mu(B(x,\delta^k))} d\mu(x).
    \end{align*}
    Now for any $\epsilon \in (0,1)$ we can find $\{u_i: 1 \leq i \leq M \} \subseteq B(y, \delta^k + C_1 \delta^n)$, where $M \leq D$ such that 
    \[
    B(y, \delta^k + C_1 \delta^n) \subseteq \bigcup_{i=1}^M B(u_i, \delta^k + C_1 \delta^n)
    \]
    and if we fix $x \in B(y, \delta^k + C_1 \delta^n)$ with $j$ be the first index such that $x \in B(u_j, \delta^k + C_1 \delta^n)$, then 
    \begin{align*}
        \frac{\chi_{B(y, \delta^k + C_1 \delta^n)}(x)}{\mu(B(x, \delta^k + C_1 \delta^n))}
        &\leq (1+\epsilon) \frac{\chi_{B(y, \delta^k + C_1 \delta^n) \cap B(u_j, \delta^k + C_1 \delta^n)}(x)}{\mu(B(u_j, \delta^k + C_1 \delta^n))}\\
        &\leq (1+\epsilon) \sum_{i=1}^M \frac{\chi_{B(y, \delta^k + C_1 \delta^n) \cap B(u_i, \delta^k + C_1 \delta^n)}(x)}{\mu(B(u_i, \delta^k + C_1 \delta^n))}\\
         &\leq (1+\epsilon) \sum_{i=1}^M \frac{\chi_{B(u_i, \delta^k + C_1 \delta^n)}(x)}{\mu(B(u_i, \delta^k))}.
    \end{align*}

    Finally,
    \begin{align*}
        &\int_G \frac{\chi_{A(y)}(x)}{\mu(B(x,\delta^k))} d\mu(x) \\
        &\leq (K+1)(C_1 +1)^\epsilon \int_G \frac{\chi_{A(y)}(x)}{\mu(B(x, \delta^k + C_1 \delta^n))} \chi_{B(y, \delta^k + C_1 \delta^n)}(x) d\mu(x) \\
        &\leq (K+1)(C_1 +1)^\epsilon (1+\epsilon) \int_G \sum_{i=1}^M \frac{\chi_{A(y) \cap B(u_i, \delta^k + C_1 \delta^n)}(x)}{\mu(B(u_i, \delta^k))} d \mu(x)\\
        & \leq (K+1)(C_1 +1)^\epsilon (1+\epsilon) \int_G \sum_{i=1}^M \frac{\chi_{A(y)}(x)}{\mu(B(y, \delta^k))} \frac{\mu(B(y, \delta^k))}{\mu(B(u_i, \delta^k)) } d\mu(x)\\
        & \leq (K+1)(C_1 +1)^\epsilon (1+\epsilon) \int_G \sum_{i=1}^M \frac{\chi_{A(y)}(x)}{\mu(B(y, \delta^k))} \frac{\mu(B(u_i, 2\delta^k + C_1 \delta^n))}{\mu(B(u_i, \delta^k)) } d\mu(x)\\
        & \qquad \qquad \text{(since $B(y, \delta^k) \subseteq B(u_i, 2\delta^k + C_1 \delta^n)$)}\\
        &\leq (K+1)^2(C_1 +1)^\epsilon (1+\epsilon)(2+C_1)^\epsilon M \frac{\mu(A(y))}{\mu(B(y, \delta^k))}\\
        &\leq (K+1)^2(C_1 +1)^\epsilon (1+\epsilon)(2+C_1)^\epsilon M K_\epsilon (\delta^{n-k})^\epsilon
        \lesssim \delta^{(n-k)\epsilon}.
    \end{align*}
    Therefore,
    \[
    \int_G \frac{\chi_{A(y)}(x)}{\mu(B(x,\delta^k))} w(x) d\mu(x) \lesssim C(w)[w]_{A_1} w(y) \delta^{(n-k)\epsilon}.
    \]
\end{proof}

\noindent
Now consider  the following averages: 
\[
M_{k,n} h(s):= \frac{1}{\mu(B(s,\delta^k))} \int_{\CI(B(s,\delta^k),n)} h(g) d \mu(g); ~ h \in L^1(\CN_w).
\]

\begin{prop}\label{prop: norm est prop}
    Let $k > n > n_{r_0}$ and $1 \leq p \leq \infty$. Then for $h \in L^p(\CN_w)$ we have 
    \[
    \norm{M_{k,n}h}^p_{p,w} \lesssim C(w)[w]_{A_1} \delta^{(n-k)\epsilon} \norm{h}^p_{p,w}.
    \]
    Furthermore, if $h$ is positive and if we set $h_n= \CE_n(h)$, then we have 
    \[
    \norm{M_{k,n}h}^p_{p,w} \lesssim C(w)[w]_{A_1} \delta^{(n-k)\epsilon} \norm{h_n}^p_{p,w}.
    \]
\end{prop}

\begin{proof}
    Note that 
		\begin{align*}
			\norm{M_{k,n} h}_{p,w}^p 
            &= \int_{G}  \norm{M_{k,n} h(s)}_p^p w(s) d\mu(s).
		\end{align*}
		Furthermore, if $\frac{1}{p} + \frac{1}{p'} =1$, then we have
        \begin{align*}
			\norm{M_{k,n} h(s)}^p_p 
			&\leq \left( \frac{1}{\mu(B(s,\delta^k))} \int_{\CI(B(s,\delta^k),n)} \norm{h(y)}_p d\mu(y) \right)^p \\
			&\leq \frac{\mu(\CI(B(s,\delta^k),n))^{p/p'}}{\mu(B(s,\delta^k))^p}  \int_{\CI(B(s,\delta^k),n)} \norm{h(y)}^p_p dy~ (\text{by Minkowski's inequality}) \\
            &\leq \left( \frac{\mu(\CI(B(s,\delta^k),n))}{\mu(B(s,\delta^k))} \right)^{p/p'} \times \frac{1}{\mu(B(s,\delta^k))} \int_{\CI(B(s,\delta^k),n)} \norm{h(y)}^p_p dy\\
            &\leq (K_\epsilon C_1)^{(p-1)\epsilon} \delta^{(n-k)(p-1)\epsilon} \times \frac{1}{\mu(B(s,\delta^k))} \int_{\CI(B(s,\delta^k),n)} \norm{h(y)}^p_p dy\\
            & \qquad \text{( by \cref{eq: annular inclusion} and \Cref{lem: annular decay} )}
		\end{align*}
        Therefore,
        \begin{align*}
            \norm{M_{k,n} h}_{p,w}^p
            \leq (K_\epsilon C_1)^{(p-1)\epsilon} \delta^{(n-k)(p-1)\epsilon} \int_{G}\left( \frac{1}{\mu(B(s,\delta^k))} \int_{\CI(B(s,\delta^k),n)} \norm{h(y)}^p_p dy \right) w(s) d\mu(s).
        \end{align*}
        Now, for $s \in G$, recall that  $A(s)= B(s,\delta^k + C_1\delta^n)   \setminus  B(s,\delta^k - C_1\delta^n)$ and we note that $ y \in A(s) \Longleftrightarrow s \in A(y)$. Therefore, we   further obtain
        \begin{align*}
            &\int_{G}\left( \frac{1}{\mu(B(s,\delta^k))} \int_{\CI(B(s,\delta^k),n)} \norm{h(y)}^p_p dy \right) w(s) d\mu(s)\\
            &\leq \int_{G}\left( \frac{1}{\mu(B(s,\delta^k))} \int_{A(s)} \norm{h(y)}^p_p dy \right) w(s) d\mu(s)\\
            &=\int_G \norm{h(y)}^p_p \left( \int_G \frac{\chi_{A(y)}(s)}{\mu(B(s,\delta^k))} w(s) d\mu(s) \right) dy~ \text{(using Fubini)}\\
            &\lesssim C(w)[w]_{A_1}  \delta^{(n-k)\epsilon} \int_G \norm{h(y)}^p_p w(y) d\mu(y)~ \text{(by \Cref{lem: norm est prop})}\\
            &=C(w)[w]_{A_1}  \delta^{(n-k)\epsilon} \norm{h}_{p,w}^p.
        \end{align*}
Thus, we finally obtain that  
        \[
        \norm{M_{k,n} h}_{p,w}^p
            \lesssim C(w)[w]_{A_1}\delta^{(n-k)\epsilon} \norm{h}_{p,w}^p.
            \]
This completes the proof.        
\end{proof}

\subsubsection{Estimate of bad diagonal part}
Let us first prove \Cref{thm: main thm weak estimate} for $b_d$. 
First, we note the following basic fact. For a given element $y$, there exists a partial isometry $u$  such that 
\begin{equation}\label{equi}
    \chi_E(\abs{y^*}) = u \chi_E(\abs{y}) u
\end{equation}
where $E$ is a Borel set in $\R$. Indeed, $u$ is the partial isometry associated with the polar decomposition: $ y = u \abs{y}$. 

\medskip
\noindent
Now recall the projection $\zeta$ defined in \cref{eq: defn zeta}, and observe the decomposition:
\begin{equation*}
		Tb_d = (1-\zeta) Tb_d (1-\zeta) + \zeta Tb_d (1-\zeta) + (1-\zeta) Tb_d \zeta + \zeta Tb_d \zeta.
\end{equation*}
We use \cref{equi} and   \Cref{lem: zeta estimate},  to obtain the following; 
	\begin{align*}
	\tilde{\varphi}_w(\abs{Tb_d}> \lambda/3) 
		&\lesssim \varphi_w (1 - \zeta) + \tilde{\varphi}_w ( \abs{\zeta Tb_d \zeta}> \lambda/9)\\
		&\lesssim [w]_{A_1}^2\frac{\norm{f}_{1,w}}{\lambda} + \tilde{\varphi}_w ( \abs{\zeta Tb_d \zeta}> \lambda/9).
	\end{align*}

\noindent 
Hence, it remains to estimate the term $\tilde{\varphi}_w ( \abs{\zeta Tb_d \zeta}> \lambda/9)$. Before proceeding with this, we recall the following almost orthogonality result from \cite{jones2003oscillation}, \cite{hong2017vector}.

\begin{lem}\label{lem: almost orthogonality}
Let $S_{k}$ be a sub-additive operator on $L_2$ for each $k\in \Z$ and $h\in L_2$.
If $(u_{n})_{n\in\Z}$ and $(v_{n})_{n\in\Z}$ are two sequences of functions in $L_{2}$ such that $h=\sum\limits_{n\in\Z}u_{n}$ and $\sum\limits_{n\in\Z}\|v_{n}\|_{2}^{2}<\infty$, then
$$\sum_{k}\|S_{k}h\|_{2}^{2}\leq w^{2}\sum_{n\in\Z}\|v_{n}\|_{2}^{2}$$
provided that there is a sequence $(\sigma(j))_{j\in \Z}$ of positive numbers with $w=\sum\limits_{j\in\Z}\sigma(j)<\infty$ such that
$$\|S_{k}(u_{n})\|_{2}\leq \sigma(n-k)\|v_{n}\|_{2}$$
for every $n,k$.
\end{lem}

\begin{lem}\label{lem:C8-1}
    For $k \leq n$ and $s \in G$
    \[
    \zeta(s) (A_{\delta^k}- \CE_k)b_{d,n}(s) \zeta(s)=0.
    \]
\end{lem}

\begin{proof}
    For $s \in G$ we can always choose $\alpha(s) \in I_n$ such that 
    \[
    d(s, z^n_{\alpha(s)})< C_0 \delta^n,~ (\text{see }~   \Cref{prop: hytonen lem}).
    \]
Now we note that $B(s,\delta^k) \subset \tilde{Q}^n_{\alpha(s)}$. Indeed, 
suppose $ t \in B(s,\delta^k)$, then observe the following 
\begin{align*}
& d(s,t)< \delta^k \\
\implies & d(t, z^n_{\alpha(s)})< \delta^k + C_0 \delta^n < (3C_1 + 1) \delta^{n+1}\\
\implies & t \in \tilde{Q}^n_{\alpha(s)}, 
\end{align*}
Thus, we have $B(s,\delta^k) \subset \tilde{Q}^n_{\alpha(s)}$. Now for  $t \in \tilde{Q}^n_{\alpha(s)}$, we observe that $\zeta(s) p_n(t)= p_n(t) \zeta(s)=0$ for all $n \in \N$. Consequently, 
it follows that $$\zeta(s) b_{d,n}(t) \zeta(s)=0, \text{ for all } t \in \tilde{Q}^n_{\alpha(s)}.$$

    Now, note that 
    \begin{align*}
     \zeta(s) A_{\delta^k}b_{d,n}(s) \zeta(s)
    &= \zeta(s) \left( \frac{1}{\mu(B(s,\delta^k))} \int_{B(s,\delta^k)} b_{d,n}(t) d \mu(t) \right) \zeta(s)\\ 
    &=\zeta(s) \left( \frac{1}{\mu(B(s,\delta^k))} \int_{B(s,\delta^k)} b_{d,n}(g) \chi_{\{t \notin \tilde{Q}^n_{\alpha(s)} \} } d \mu(g) \right) \zeta(s)\\
    &=0.
    \end{align*}


Furthermore,
\begin{align*}
    \zeta(s) \CE_k b_{d,n}(s) \zeta(s)
&=\zeta(s) \left( \frac{1}{\mu(Q^k_{\alpha(s)}))} \int_{Q^k_{\alpha(s)}} b_{d,n}(g) d \mu(g) \right) \zeta(s)\\ 
&=\zeta(s) \left( \frac{1}{\mu(Q^k_{\alpha(s)})} \int_{Q^k_{\alpha(s)}} b_{d,n}(g) \chi_{\{g \notin \tilde{Q}^n_{\alpha(s)} \} } d \mu(g) \right) \zeta(s)\\
&=0,
\end{align*}
since, $k \leq n$ and $Q^k_{\alpha(s)} \subseteq \tilde{Q}^k_{\alpha(s)} \subseteq \tilde{Q}^n_{\alpha(s)}$.
\end{proof}

\noindent
Now the following result holds.

\begin{prop}\label{prop: bad diag estimate}
The following estimate holds true
\[
\widetilde{\varphi}_w \left( \abs{\zeta Tb_{d}\zeta} > \frac{\lambda}{9}  \right)  \lesssim C(w)[w]^2_{A_1} \frac{\norm{f}_{1,w}}{\lambda}.
\]
\end{prop}

\begin{proof}
    By Chebychev's inequality,
    \[
    \widetilde{\varphi}_w \left( \abs{\zeta Tb_{d}\zeta} > \frac{\lambda}{9}  \right)  \lesssim \frac{\norm{\zeta Tb_{d}\zeta}^2_{L_2(L^\infty(\Omega) \otimes \CN_w)}}{\lambda^2}.
    \]
    Since $\displaystyle \sum_{n \in \N} \norm{p_n}_{2,w}^2= \varphi_w(1-q) \lesssim [w]_{A_1} \frac{\norm{f}_{1,w}}{\lambda}$, it suffices to show
    \[
    \norm{\zeta Tb_{d}\zeta}^2_{L_2(L^\infty(\Omega) \otimes \CN_w)} \lesssim C(w)[w]_{A_1} \lambda^2 \sum_{n \in \N} \norm{p_n}_{2,w}^2.
    \]
 Further observe that
    \[
    \norm{\zeta Tb_{d}\zeta}^2_{L_2(L^\infty(\Omega) \otimes \CN_w)}= \sum_{k> n_{r_0}} \norm{\zeta T_kb_{d}\zeta}_{2,w}^2= \sum_{k> n_{r_0}} \norm{\zeta (A_{\delta^k}- \CE_k)b_{d}\zeta}_{2,w}^2.
    \]
    Therefore, we will show
    \[
    \sum_{k> n_{r_0}} \norm{\zeta (A_{\delta^k}- \CE_k)b_{d}\zeta}_{2,w}^2 \lesssim C(w)[w]_{A_1} \lambda^2 \sum_{n > n_{r_0}} \norm{p_n}_{2,w}^2.
    \]
    Set $b_{d,n}:= p_n(f-f_n)p_n$ and note that $b_d= \displaystyle  \sum_{n \in \N} b_{d,n}$. Then by applying \Cref{lem: almost orthogonality}, with $S_kb_d= \zeta (A_{\delta^k}- \CE_k)b_d\zeta$, $u_n= b_{d,n}$ and $v_n= p_n$ we find that it is enough to prove 
    \[
    \norm{\zeta (A_{\delta^k}- \CE_k)b_{d,n}\zeta}_{2,w}^2 \lesssim C(w)[w]_{A_1} \delta^{- \abs{n-k} \epsilon}\lambda^2\norm{p_{n}}^2_{2,w}.
    \]
Now we subdivide the proof into two cases. For $k \leq n$, by \Cref{lem:C8-1}, we obtain that the left hand side of the above equation is $0$. So we now consider the case $k>n$. In this case, note that by \Cref{thm: CZ decomp} we have $\CE_k(b_{d,n})= \CE_k (\CE_n b_{d,n})=0$. Therefore, it is enough to show that 
\[
\norm{\zeta A_{\delta^k} b_{d,n} \zeta}_{2,w} \lesssim \sqrt{C(w)[w]_{A_1}} \delta^{(n-k)\epsilon/2} \lambda \norm{p_n}_{2,w}.
\]
Furthermore, 
since 
\[
\int_{Q^n_\alpha} b_{d,n}(g) d\mu(g)
= \int_{Q^n_\alpha} \CE_n(b_{d,n})(g) d \mu(g)
=0, ~\text{by  \Cref{thm: CZ decomp}}, 
\]
we observe that 
\begin{align*}
    A_{\delta^k} b_{d,n}(s)
&= \frac{1}{\mu(B(s,\delta^k))} \int_{B(s,\delta^k)} b_{d,n}(g) d \mu(g)\\
&= \sum_{ Q }  \frac{1}{\mu(B(s,\delta^k))} \int_{B(s,\delta^k) \cap Q} b_{d,n}(g) d \mu(g) \\
&= \frac{1}{\mu(B(s,\delta^k))} \int_{\CI(B(s,\delta^k),n)} b_{d,n}(g) d \mu(g) \\
&= M_{k,n} b_{d,n}(s).
\end{align*}

Note that $p_n \CE_n(f) p_n \lesssim \lambda p_n$ for all $n \in \N$. Now since $b_{d,n}= p_nfp_n- p_nf_np_n$ and $\CE_n(p_n f p_n)= p_n f_n p_n$ for all $n \in \N$, by \Cref{prop: norm est prop}, we observe that 
\[
\norm{A_{\delta^k} b_{d,n}}_{2,w} \lesssim \sqrt{C(w)[w]_{A_1} \delta^{(n-k) \epsilon}} \norm{p_n f_n p_n}_{2,w} \lesssim \sqrt{C(w)[w]_{A_1}}\lambda \delta^{(n-k) \epsilon/2} \norm{p_n}_{2,w}.
\]
This proves the result.
\end{proof}

\subsubsection{Estimate of bad off-diagonal part}
First consider the decomposition:
\begin{equation*}
		Tb_{off} = (1-\zeta) Tb_{off} (1-\zeta) + \zeta Tb_{off} (1-\zeta) + (1-\zeta) Tb_{off} \zeta + \zeta Tb_{off} \zeta.
\end{equation*}
Then, by applying \Cref{lem: zeta estimate}, we again reduce the problem to estimating the term $\tilde{\varphi}_w ( \abs{\zeta Tb_{off} \zeta}> \lambda/9)$. Before we proceed, we need the following crucial result.
	
\begin{lem}\label{lem: bad part est prop}
Let $s \in G$ and $ k>n >n_{r_0}$. Then the following estimate holds;
    \begin{align*}
        \norm{\int_{\CI(B(s,\delta^k),n)} b_n^{off} (t) d\mu(t)}_1
\lesssim \lambda \int_{\CI_1(B(s,\delta^k),n)} \tau(p_n(t)) d \mu(t).
    \end{align*}
\end{lem}

\begin{proof}
   Recall that   $b_n^{off}= p_n(f- f_n) q_n + q_n(f- f_n) p_n
        =p_nfq_n + q_nfp_n$, (see \Cref{thm: Cuculescu}(1)), it is enough to prove the result for $b_n^{off}= p_nfq_n$. First, applying \Cref{prop:Holder-type-ineq} with $ \frac{1}{1} = \frac{1}{1}+ \frac{1}{\infty}$, we note that
\begin{align*}
    \norm{\int_{B(s, \delta^k) \cap Q^n_\alpha} (p_nfq_n) (t) d\mu(t)}_1
   & \leq \norm{\int_{ Q^n_\alpha} p_{Q^n_\alpha}f(t)q_{Q^n_\alpha} d\mu(t)}_1\\
    &\leq \norm{\left(\int_{Q^n_\alpha} p_{Q^n_\alpha}f(t)p_{Q^n_\alpha} \right)^{1/2} }_1
    \norm{\left(\int_{ Q^n_\alpha} q_{Q^n_\alpha}f(t)q_{Q^n_\alpha} \right)^{1/2} }_\infty\\
    &=\norm{\int_{Q^n_\alpha} p_{Q^n_\alpha}f(t)p_{Q^n_\alpha}  }_{1/2}^{1/2}
    \norm{\int_{ Q^n_\alpha} q_{Q^n_\alpha}f(t)q_{Q^n_\alpha}  }^{1/2}_\infty.
\end{align*}
Also,
\begin{align*}
    \norm{\int_{Q^n_\alpha} p_{Q^n_\alpha}f(t)p_{Q^n_\alpha}  }_{1/2}
    &=\norm{p_{Q^n_\alpha} \int_{Q^n_\alpha} p_{Q^n_\alpha}f(t)p_{Q^n_\alpha}  }_{1/2}\\
    &\leq \norm{p_{Q^n_\alpha}}_1  \norm{\int_{Q^n_\alpha} p_{Q^n_\alpha}f(t)p_{Q^n_\alpha}}_1~ \text{by Holder's inequality with } \frac{1}{1/2}= \frac{1}{1} + \frac{1}{1}.
\end{align*}
Now since $q_{Q^n_\alpha}f(t)q_{Q^n_\alpha} \lesssim \lambda $ and $p_{Q^n_\alpha}f(t)p_{Q^n_\alpha} \lesssim \lambda p_{Q^n_\alpha}$, we obtain
\begin{align*}
    \norm{\int_{B(s, \delta^k) \cap Q^n_\alpha} (p_nfq_n) (t) d\mu(t)}_1
    &\lesssim \norm{p_{Q^n_\alpha}}^{1/2}_1  \norm{\int_{Q^n_\alpha} p_{Q^n_\alpha}f(t)p_{Q^n_\alpha}}^{1/2}_1 (\lambda \mu(Q^n_\alpha))^{1/2}\\
    &\lesssim \tau(p_{Q^n_\alpha})^{1/2} \left( \lambda \tau(p_{Q^n_\alpha}) \mu(Q^n_\alpha) \right)^{1/2} (\lambda \mu(Q^n_\alpha))^{1/2}\\
    &= \lambda \tau(p_{Q^n_\alpha}) \mu(Q^n_\alpha)\\
    &= \lambda \varphi(p_n \chi_{Q^n_\alpha}).
\end{align*}
Furthermore,
\begin{align*}
    \norm{\int_{\CI(B(s,\delta^k),n)} (p_nfq_n) (t) d\mu(t)}_1
    &\leq  \sum_{\alpha \in I_n,~ \partial(B(s,\delta^k)) \cap Q^n_\alpha \neq \emptyset} \norm{\int_{B(s, \delta^k) \cap Q^n_\alpha} (p_nfq_n) (t) d\mu(t)}_1\\
    &\lesssim \lambda \sum_{\alpha \in I_n,~ \partial(B(s,\delta^k)) \cap Q^n_\alpha \neq \emptyset} \varphi(p_n \chi_{Q^n_\alpha})\\
    &= \lambda \varphi(p_n \chi_{\CI_1(B(s,\delta^k),n)})\\
    &=\lambda \int_{\CI_1(B(s,\delta^k),n)} \tau(p_n(t)) d \mu(t).
\end{align*}
\end{proof}

\begin{prop}\label{bad part est prop}
		The following estimate holds true;
		\begin{equation}
			\tilde{\varphi}_w ( \abs{\zeta Tb_{off} \zeta}> \lambda/9) \lesssim C(w)[w]_{A_1}  \frac{\norm{f}_{1,w}}{\lambda}.
		\end{equation}
	\end{prop}

    \begin{proof}
        Applying Chebychev's inequality we obtain 
        \begin{align*}
            \lambda \tilde{\varphi}_w ( \abs{\zeta Tb_{off} \zeta}> \lambda/9) 
            &\lesssim \norm{\zeta Tb_{off} \zeta}_{1,w}\\
            &\leq \sum_{n} \sum_k \norm{\zeta \epsilon_k (A_{\delta^k}- \CE_k) b_n^{off} \zeta}_{1,w}\\
			&\leq \sum_{n} \sum_k \norm{\zeta (A_{\delta^k} - \CE_k) b_n^{off} \zeta}_{1,w},
        \end{align*}
        where, for all $n \in \Z$
        \[
        b_n^{off}= p_n(f- f_n) q_n + q_n(f- f_n) p_n
        =p_nfq_n + q_nfp_n.
        \]
        By an argument similar to \Cref{lem:C8-1}, it follows that 
        \[
        \zeta (A_{\delta^k} - \CE_k) b_n^{off} \zeta=0
        \]
        for $k \leq n$. For $k>n$, first note that 
        \[
        \CE_k(b_n^{off})= \CE_k \left( \CE_n(b_n^{off}) \right)=0
        \]
        follows from \Cref{thm: CZ decomp}. Therefore, in this case, it is enough to show that 
        \[
        \sum_{n} \sum_k \norm{\zeta A_{\delta^k} b_n^{off} \zeta}_{1,w} \lesssim C(w)[w]_{A_1} \norm{f}_{1,w}
        \]
        Furthermore, observe that 
\begin{align*}
    A_{\delta^k} b_n^{off}(s)
&= \frac{1}{\mu(B(s,\delta^k))} \int_{B(s,\delta^k)} b_n^{off}(g) d \mu(g)\\
&= \frac{1}{\mu(B(s,\delta^k))} \int_{\CI(B(s,\delta^k),n)} b_n^{off}(g) d \mu(g)\\
&= M_{k,n} b_n^{off}(s),
\end{align*}
since 
\[
\int_{Q^n_\alpha} b_n^{off}(g) d\mu(g)
= \int_{Q^n_\alpha} \CE_n(b_n^{off})(g) d \mu(g)
=0, ~\text{by  \Cref{thm: CZ decomp}}.
\]

\noindent
Hence, for $s \in G$, recall that  $A(s)= B(s,\delta^k + C_1\delta^n)   \setminus  B(s,\delta^k - C_1\delta^n)$ and from \Cref{lem: bad part est prop}  we  obtain
\begin{align*}
\norm{M_{k,n} b_n^{off}}_{1,w}
&=\int_G \norm{M_{k,n} b_n^{off}(s)}_1 w(s) d\mu(s)\\
&\lesssim \lambda \int_G \left( \frac{1}{\mu(B(s,\delta^k))} \int_{\CI_1(B(s,\delta^k),n)} \tau(p_n(g)) d \mu(g) \right) w(s) d\mu(s)\\
&\leq \lambda \int_G \left( \frac{1}{\mu(B(s,\delta^k))} \int_{A(s)} \tau(p_n(g)) d \mu(g) \right) w(s) d\mu(s) \\
&= \lambda \int_G \tau(p_n(g)) \left( \int_G \frac{\chi_{A(g)}(s)}{\mu(B(s,\delta^k))} w(s) d\mu(s) \right) d\mu(g)~ \text{(using Fubini)}\\
            &\lesssim C(w)[w]_{A_1}  \delta^{(n-k)\epsilon} \lambda  \int_G \tau(p_n(g)) w(g) d\mu(g)~ \text{(by \Cref{lem: norm est prop})}\\
            &= \lambda C(w)[w]_{A_1}  \delta^{(n-k)\epsilon} \varphi_w(p_n).
\end{align*}
So,
\begin{align*}
    \sum_{n} \sum_{k>n} \norm{\zeta A_{\delta^k} b_n^{off} \zeta}_{1,w} 
    &\lesssim \sum_{n} \sum_{k>n} \lambda C(w)[w]_{A_1}  \delta^{(n-k)\epsilon} \varphi_w(p_n)\\
    &\leq \lambda C(w)[w]_{A_1}  \sum_n \varphi_w(p_n)\\
    &= \lambda C(w)[w]_{A_1}  \varphi_w(1-q) \\ 
    &\lesssim C(w)[w]_{A_1} \norm{f}_{1,w}.
\end{align*}
\end{proof}

\subsubsection{Estimate of good part}

In this section, we will prove that $T$ is bounded from $L_{2}(\mathcal N_w)$ to $L_2(L_{\infty}(\Omega)\otimes \CN_w)$. To begin with we first recall the following result.

For $k \leq n$ and $\alpha \in I_n$, let us define the set
\[
\mathfrak{H}(B_{\delta^k}, Q^n_\alpha) := \left\{ x \in Q^n_\alpha : B(x, \delta^k) \cap (Q^n_\alpha)^c \neq \emptyset \right\}.
\]
Suppose $\CN(\alpha)= \{ \beta \in I_n:   Q_\beta^n  \text{ is a neighbour of } Q_\alpha \}  $ and  we assume that $ \alpha \in \CN(\alpha)$. Then  
We have the following useful lemma.

\begin{prop}\label{prop: measure of a good point set}\cite[Lemma 2.7]{hong2021quantitative}
    There exists constants $\eta>0$ and $n_1 \in \N$ depending only on the group $G$ such that for all $n \geq k > n_1$, the following estimate holds.
    \[
    \mu(\mathfrak{H}(B_{\delta^{k}},Q^{n}_\alpha))\lesssim \delta^{(k-n)\eta} \mu(Q^{n}_\alpha).
    \]
\end{prop}

Now we prove the following strong $(2,2)$ type inequality.

\begin{prop}
        For every $n \in \N$ and $\alpha \in I_n$, the number of cubes $Q^n_\beta$ such that $d(z^n_\alpha, z^n_\beta)< (2C_1 +1) \delta^n$ is bounded by a universal constant.
\end{prop}

\begin{proof}
    Let us choose $m\geq 0$ such that
    \begin{align*}
        2^{m-1} < \frac{2C_1+1+c_0}{c_0} \leq 2^m.
    \end{align*}
    Let $N$ be the cardinality of the set $J:=\{\beta \in I_n: d(z^n_\alpha, z^n_\beta) < (2C_1 +1) \delta^n \}$. Then observe that since the cubes $Q^n_\beta$ are disjoint, we have
    \begin{align*}
        N\min_{\beta \in J} \mu(B(z^n_\beta, c_0 \delta^n)) \leq \sum_{\beta \in J} \mu(B(z^n_\beta, c_0 \delta^n)) \leq \mu(B(z^n_\alpha, (2C_1 + c_0 +1)\delta^n)).
    \end{align*}
    Also by the measure doubling  condition, we have 
    $$\mu(B(z^n_\alpha, (2C_1 + c_0 +1)\delta^n)) \leq \mu(B(z^n_\alpha, 2^mc_0\delta^n)) \leq C^m \mu(B(z^n_\alpha, c_0 \delta^n))$$
    Furthermore, notice that 
    \begin{align*}
        &B(z^n_\alpha, c_0 \delta^n) \subseteq B(z^n_\beta, (2C_1+1+c_0) \delta^n) \subseteq B(z^n_\beta, 2^mc_0 \delta^n)\\
        \implies& \mu(B(z^n_\alpha, c_0 \delta^n)) \leq C^m \mu(B(z^n_\beta, c_0 \delta^n)) \\
        \implies& C^{-m} \mu(B(z^n_\alpha, c_0 \delta^n)) \leq   \mu(B(z^n_\beta, c_0 \delta^n))
    \end{align*}
Therefore, we obtain the following: 
\begin{align*}
NC^{-m} \leq C^{m} \implies N \leq C^{ 2m}. 
\end{align*}
\end{proof}

\begin{prop}\label{compare-1}
    Let $n \in \N$ and $\alpha \in I_n$ and $U^n_\alpha$ be the union of $Q^n_\alpha$ and all cubes $Q^n_\beta$ such that $d(z^n_\alpha, z^n_\beta)< (2C_1 +1)\delta^n$. Then we have the following.
    \begin{enumerate}
        \item For all $s \in Q^n_\alpha$, the ball $B(s, \delta^k) \subseteq U^n_\alpha$, when $k \leq n $.\\
        \item Suppose  $ Q_\beta^{n-1} \in U^n_\alpha$, then 
        $ w(U^n_\alpha) \lesssim [w]_{ A_2} w( Q_\beta^{n-1})$.
    \end{enumerate}
\end{prop}

\begin{proof}
    For the first part, let $y$ be an element of $G$ such that $d(s,y)<\delta^k$. Also note that $Q^n_\beta \subseteq B(z^n_\beta, C_1\delta^n)$ for all $\beta \in I_n$. Since union of $Q^n_\beta$, where $\beta \in I_{n}$ exhaust $G$, so we can assume that $y \in Q^n_\beta$ for some $\beta$. Then we must have
    \begin{align*}
        d(z^n_\alpha, z^n_\beta)
        \leq d(z^n_\alpha,s)+ d(s,y) + d(y, z^n_\beta) < (2C_1 +1)\delta^n.
    \end{align*}
    Therefore, we have $y \in U^n_\alpha$.\\

    For the second part, let $ d= 3C_1  +1$. 
    Then we note that 
    $$ U_\alpha^n \subseteq B( z_\beta^{n-1}, 2d \delta^n )$$
    Therefore, we observe that 
    \begin{align*}
     w(U^n_\alpha) &\leq w(  B( z_\beta^{n-1}, 2d \delta^n ) )\\
     & \leq   [w]_{ A_2} \left(  \frac{\mu( B( z_\beta^{n-1}, 2d \delta^n ))}
   {\mu( Q_\beta^{ n-1}) }\right)^2 w( Q_\beta^{n-1})\\
   &\leq [w]_{ A_2} \left(  \frac{\mu( B( z_\beta^{n-1}, 2d \delta^n ))}
   {\mu( B( z_\beta^{n-1}, c_0\delta^n))}\right)^2 w( Q_\beta^{n-1})\\
   & \lesssim [w]_{ A_2} w( Q_\beta^{n-1})
    \end{align*}
\end{proof}

\begin{lem}\label{lem: strong estimate} Suppose 
    $h \in L^2(\CN_w)$, then   
    \[
    \norm{Th}_{L^2(L_{\infty}(\Omega)\otimes\CN_w)}\lesssim C(w)(\max\{1, [w]_{A_1} \}) \norm{h}_{2,w}
    \]
\end{lem}

\begin{proof}
     Observe that it follows from \cite[Lemma 6.7]{fan2025weighted} that 
\[
\sum_{n}\norm{dh_{n}}_{2,w}^{2}=\norm{h}_{2,w}^{2},
    \]
    where $ h_n = \CE_n(h)$,  $dh_n= h_{n}- h_{n+1}$ and $h= \sum_n dh_n$.

    Furthermore, by \Cref{lem: almost orthogonality}, it suffices to prove
\begin{align*}
\norm{(A_{\delta^k}-\CE_k)dh_{n}}^{2}_{2,w}\lesssim C(w)\delta^{-|n-k|}\|dh_{n} \|^{2}_{2,w}.
\end{align*}
Consider the case $k \leq n$. In this case note that $\CE_k (dh_{n})=dh_{n}$. Therefore, it is enough to show
\begin{align*}
\norm{A_{\delta^k}dh_{n}-dh_n}^{2}_{2,w}\lesssim C(w) \delta^{k-n}\norm{dh_n}^{2}_{2,w}.
\end{align*}
Note that 
\begin{align*}
    \norm{A_{\delta^k}dh_{n}-dh_n}^{2}_{2,w}
    &= \int_G \norm{A_{\delta^k}dh_{n}(s)-dh_n(s)}^{2}_{2} w(s) d\mu(s)\\
    &= \sum_{\alpha \in I_n} \int_{Q^n_\alpha} \norm{A_{\delta^k}dh_{n}(s)-dh_n(s)}^{2}_{2} w(s) d\mu(s).
\end{align*}
Since \( dh_n \) is a constant operator on each dyadic cube \( Q^n_\alpha \) for all \( \alpha \in I_n \), we have
\[
A_{\delta^k} dh_n(x) - dh_n(x) = 0 \quad \text{if} \quad B(x, \delta^k) \subseteq Q^n_\alpha.
\]
Therefore, for \( x \in Q^n_\alpha \), the term \( (A_{\delta^k} dh_n - dh_n)(x) \) may be nonzero only when the ball \( B(x, \delta^k) \) intersects the complement of the cube \( Q^n_\alpha \). Motivated by this, for a given dyadic cube \( Q^n_\alpha \) and a ball \( B(x, \delta^k) \) in \( G \), we define the set:
\[
\mathfrak{H}(B_{\delta^k}, Q^n_\alpha) := \left\{ x \in Q^n_\alpha : B(x, \delta^k) \cap (Q^n_\alpha)^c \neq \emptyset \right\}.
\]

Let for any $\alpha \in I_n$, $U^n_\alpha$ denotes the union of $Q^n_\alpha$ and all cubes $Q^n_\beta$ such that $d(z^n_\alpha, z^n_\beta)< (2C_1 +1)\delta^n$. Furthermore, for any $E\subseteq G$, we define $m_E:= \sup_{s \in E} \norm{dh_n(s)}_2$.

Then we first observe that 
\begin{align*}
    &\int_{Q^n_\alpha} \norm{A_{\delta^k}dh_{n}(s)-dh_n(s)}^{2}_{2} w(s) d\mu(s)\\
    &\leq 2 m_{U^n_\alpha}^2 w(\mathfrak{H}(B_{\delta^k}, Q^n_\alpha))\\
    &\leq C(w) m_{U^n_\alpha}^2 w(Q^n_\alpha) \left( \frac{\mu(\mathfrak{H}(B_{\delta^k}, Q^n_\alpha))}{\mu(Q^n_\alpha)} \right)^\upsilon\\
    &\leq C(w) C_2 \delta^{(k-n) \eta \upsilon} m_{U^n_\alpha}^2 w(Q^n_\alpha),~ \text{by } \Cref{prop: measure of a good point set}
\end{align*}

Now, for all $\alpha \in I_n$, we have
\begin{align*}
m_{U^{n}_\alpha}^{2}\, w(Q^{n}_\alpha)
&\;\le\;
\left(\sum_{Q^{n-1}_\beta \subseteq U^{n}_\alpha} m_{Q^{n-1}_\beta}^2\right) w(U^{n}_\alpha)\\
&\lesssim [w]_{A_2} \sum_{Q^{n-1}_\beta \subseteq U^{n}_\alpha} m_{Q^{n-1}_\beta}^2  w(Q^{n-1}_\beta), ~ \text{ by } \Cref{compare-1}(2)\\
&= [w]_{A_2} \sum_{Q^{n-1}_\beta \subseteq U^{n}_\alpha} \int_{Q^{n-1}_\beta} \norm{dh_n(s)}_2^2 w(s) d\mu(s).
\end{align*}

Therefore, we obtain
\begin{align*}
    \norm{A_{\delta^k}dh_{n}-dh_n}^{2}_{2,w}
    &\lesssim C(w) [w]_{A_2} \delta^{(k-n) \eta \upsilon}  \sum_{\alpha\in I_n}  \sum_{Q^{n-1}_\beta \subseteq U^{n}_\alpha} \int_{Q^{n-1}_\beta} \norm{dh_n(s)}_2^2 w(s) d\mu(s)\\
    &\lesssim  C(w) [w]_{A_2} \delta^{(k-n) \eta \upsilon}  \sum_{\alpha\in I_n}  \sum_{Q^{n-1}_\beta \subseteq Q^n_\alpha} \int_{Q^{n-1}_\beta} \norm{dh_n(s)}_2^2 w(s) d\mu(s)\\
    &\lesssim C(w) [w]_{A_2} \delta^{(k-n) \eta \upsilon} \norm{dh_n}^2_{2,w}.
\end{align*}

Now consider the case $k > n$. In this case $\CE_k(dh_n)=0$. Therefore, we only need to prove
\[
\norm{A_{\delta^k}dh_{n}}^{2}_{2,w} \lesssim C(w)[w]_{A_2} \delta^{(k-n)} \norm{dh_n}_{2,w}^2.
\]
But since $A_{\delta^k}dh_{n}= M_{k, n} dh_n$, applying \Cref{prop: norm est prop}, we obtain the above inequality.


Now consider the case $k > n$. In this case $\CE_k(dh_n)=0$. Therefore, we only need to prove
\[
\norm{A_{\delta^k}dh_{n}}^{2}_{2,w} \lesssim C(w)[w]_{A_1} \delta^{(k-n)} \norm{dh_n}_{2,w}^2.
\]
But since $A_{\delta^k}dh_{n}= M_{k, n} dh_n$, applying \Cref{prop: norm est prop}, we obtain the above inequality.
\end{proof}

As a consequence we obtain the weak type $(1,1)$ estimate for the good part. 

	\begin{thm}\label{thm: weak estimate of good part}
		The following is true.
		\begin{align*}
			\lambda \tilde{\varphi}_w(\abs{Tg}> \lambda) \lesssim  \left( C(w) \max\{1, [w]_{A_1}\} \right)^2 \norm{f}_{L^1(\CN_w)}
		\end{align*}
	\end{thm}

    \begin{proof}
        By Chebychev's inequality, we observe that 
		\begin{align*}
			\tilde{\varphi}_w(\abs{Tg}> \lambda)
            &\leq \norm{Tg}_{L_2(L_{\infty}(\Omega)\otimes\CN_w)}\\
			&\lesssim C(w)(\max\{1, [w]_{A_1} \}) \frac{ \norm{g}^2_{L^2 (\CN_w)}}{\lambda^2},~ \text{by } \Cref{lem: strong estimate}\\
			&\lesssim C(w)(\max\{1, [w]_{A_1} \})  \frac{\norm{g}_{L^1 (\CN_w)}}{\lambda}, ~ \text{since } \norm{g}_\infty \lesssim \lambda\\
            & \lesssim \left( C(w) \max\{1, [w]_{A_1}\} \right)^2 \frac{\norm{f}_{L^1 (\CN_w)}}{\lambda}~ \text{by } \Cref{thm: CZ decomp}.
		\end{align*} 
    \end{proof}

\subsection{Strong type $(p,p)$ estimate}

To prove the strong type $(p,p)$ estimate, essentially we need to estimate the strong type $(2,2)$ bounds of the operator $T$ in \cref{eq: linearization}.    

   		\begin{thm}\label{strongbdd}
		Let $1<p < \infty$. Then the following is true
		\begin{equation}
\left\Vert (T_k f) \right\Vert_{L^p(\CN_w;\ell^{rc}_2)} \lesssim C(w)  \left\Vert f \right\Vert_{p,w}~ \forall f \in L^p(\CN_w).
\end{equation}   
		\end{thm}
		\begin{proof}

We now conclude the argument by appealing to interpolation and duality. By \Cref{thm: main thm weak estimate}, the operator \( T \) is of weak type \((1,1)\), mapping \( L^1(\mathcal{N}_w, \varphi_w) \) into \( L^{1,\infty}(L^\infty(\Omega) \otimes \mathcal{N}_w, \tilde{\varphi}_w) \). In addition, it follows from \Cref{lem: strong estimate} that \( T \) is bounded from \( L^2(\mathcal{N}_w, \varphi_w) \) to \( L^2(L^\infty(\Omega) \otimes \mathcal{N}_w, \tilde{\varphi}_w) \).

Applying the non-commutative Marcinkiewicz interpolation theorem, we deduce that \( T \) extends to a bounded operator from \( L^p(\mathcal{N}_w, \varphi_w) \) to \( L^p(L^\infty(\Omega) \otimes \mathcal{N}_w, \tilde{\varphi}_w) \) for all \( 1 < p < 2 \). By duality, it follows that the adjoint operator \( T^* \) is bounded on \( L^{p'} \), where \( \frac{1}{p} + \frac{1}{p'} = 1 \). A straightforward verification shows that \( T^* = T \), and hence the boundedness of \( T \) extends to all \( 1 < p < \infty \).
\end{proof}

\section{Quantitative Ergodic Theorem}\label{sec: Quant erg thm}

In this section we establish the quantitative mean ergodic theorems for actions of group metric measure spaces $(G,d, \mu)$ by invertible power bounded operators on non-commutative $L^p(\CM, \tau)$ for $1<p< \infty$.  We further need to assume that the metric is left $G$-invariant. Before we begin, let us first fix some notations.
Let $\alpha=(\alpha_g)$ be an action of $(G,d, \mu)$ on $L^p(\CM, \tau)$ such that 
\begin{equation*}
    \sup_{g \in G} \norm{\alpha_g: L^p(\CM, \tau) \to L^p(\CM, \tau)}< \infty.
\end{equation*}
Let $r>0$, then consider 
\begin{equation*}
    M_r(x):= \frac{1}{\mu(B_r)} \int_{B_r} \alpha_g(x) d \mu(g); ~ x \in L^p(\CM, \tau).
\end{equation*}
Now for an increasing sequence $(r_i)_{i \in \N}$ of positive real numbers, set for $i \in \N$
\begin{equation*}
    T_i:= M_{r_i}- M_{r_{i+1}}.
\end{equation*}

The non-commutative quantitative ergodic theorem states the following.

\begin{thm}\label{thm: quant erg thm}
    Let $1<p< \infty$ and $\alpha=(\alpha_g)$ be an action of $(G,d, \mu)$ on $L^p(\CM, \tau)$ by invertible power bounded operators. There exists a constant $C_p>0$ such that for all $x \in L^p(\CM, \tau)$
    \begin{equation*}
        \sup  \norm{(T_i(x))_{i \in \N}}_{L^p(\CM; \ell_2^{rc})}  \leq C_p \norm{x}_{L^p(\CM, \tau)},
    \end{equation*}
    where the supremum is considered over all increasing sequence of strictly positive numbers.
\end{thm}

To prove \Cref{thm: quant erg thm}, we first need to establish a transference principle. Let us first fix some relevant notations.

Let $1\leq p< \infty$ and \( (G, d, \mu) \) be a group metric measure space. Let \( f: G \to L^p(\mathcal{N}, \varphi) \) be a function for which the averaging operator is defined as
\[
A'_r(f)(h) := \frac{1}{\mu(B_r)} \int_{B_r} f(hg) \, d\mu(g), \quad r > 0,~ h \in G.
\]

 We remark that since the metric $d$ is $G$-invariant, we observe that $A_r'(f)= A_r(f)$, where $A_r(f)$, for $r>0$, is defined in \cref{avg op}. 

Similar to the previous case, for an increasing sequence $(r_i)_{i \in \N}$ of positive real numbers, set
\begin{equation*}
    T'_i:= A'_{r_i}- A'_{r_{i+1}}
\end{equation*}

\begin{thm}\label{thm: tranf princ}
Let $(r_i)_{i \in \N}$ be an increasing sequence of positive real numbers and $1\leq p< \infty$. If there exists a positive constant $C_p$ satifying
\begin{equation*}
    \norm{(T'_if)_{i \in \N}}_{L^p(\CN; \ell_2^{rc})} \leq C_p \norm{f}_{L^p(\CN, \varphi)} ~ \text{for all } f \in L^p(\CN, \varphi)
\end{equation*}
Then there exists a constant $C>0$ such that
\begin{equation*}
    \norm{(T_ix)}_{L^p(\CM; \ell_2^{rc})} \leq CC_p \norm{x}_{L^p(\CM, \tau)}.
\end{equation*}
\end{thm}

\begin{proof}
    Let $x \in L^p(\CM, \tau)$. It is enough to prove that for any $i_0 \in \N$
    \begin{align*}
         \norm{(T_ix)_{1 \leq i \leq i_0}}_{L^p(\CM; \ell_2^{rc})} \lesssim \norm{x}_{L^p(\CM, \tau)}.
    \end{align*}
    Consider a compact set $K$ containing the largest ball $B_{r_{i_0}}$. Let $\epsilon>0$ and choose a $R>0$ such that $\frac{\mu(KB_R)}{\mu(B_R)} \leq 1 + \epsilon$. Now define a function $f:G \to L^p(\CM)$ by
    \begin{align*}
        f(h):= 1_{KB_R}(h) \alpha_h(x)
    \end{align*}
    and observe that for all $s \in B_R$ and $1 \leq i \leq i_0 + 1$,
    \begin{align*}
        \alpha_s M_{r_i}(x)= A'_{r_i} f(s).
    \end{align*}
    which implies 
    \begin{align*}
        \alpha_s T_{i}(x)= T'_{i} f(s).
    \end{align*}
    Now since $\sup_{s \in G} \norm{\alpha_s}< \infty$, by Lemma \cite[Lemma 1.8]{hong2024quantitative}, we obtain for all $s \in B_R$
    \begin{align*}
        \norm{(T_ix)_{1\leq i \leq i_0}}_{L^p(\CM; \ell_2^{rc})} \lesssim \norm{\alpha_s((T_ix))_{1\leq i \leq i_0}}_{L^p(\CM; \ell_2^{rc})} = \norm{(T'_if(s))_{1\leq i \leq i_0}}_{L^p(\CM; \ell_2^{rc})}.
    \end{align*}
    Furthermore, since $L^p(\CM \otimes L^\infty(G); \ell_2^{rc}) \approx L^p(G; L^p(\CM; \ell_2^{rc})) $, we have 
    \begin{align*}
        \norm{(T_ix)_{1\leq i \leq i_0}}^p_{L^p(\CM; \ell_2^{rc})}
        &\lesssim \frac{1}{\mu(B_R)} \int_{B_R} \norm{(T'_if(s))_{1\leq i \leq i_0}}^p_{L^p(\CM; \ell_2^{rc})} d\mu(s)\\
        &\leq \frac{1}{\mu(B_R)} \int_{G} \norm{(T'_if(s))_{1\leq i \leq i_0}}^p_{L^p(\CM; \ell_2^{rc})} d\mu(s)\\
        &\lesssim \frac{1}{\mu(B_R)} \norm{(T'_if)_{1\leq i \leq i_0}}^p_{L^p(\CM \otimes L^\infty(G); \ell_2^{rc})}\\
        &\leq \frac{C_p}{\mu(B_R)} \norm{f}^p_{L^p(\CM \otimes L^\infty(G))}\\
        &\leq C_p \frac{\mu(KB_R)}{\mu(B_R)} \norm{x}^p_{L^p(\CM)}\\
        &\leq C_p(1+\epsilon) \norm{x}^p_{L^p(\CM)}.
    \end{align*}
Thus, we get the result.
\end{proof}

Now that we have the transference principle in our hand, we just need to prove the following result.

\begin{thm}\label{thm: quant main thm}
    Let $1\leq p< \infty$ and $(r_i)_{i \in \N}$ be an increasing sequence of positive numbers. Then the following are true with an universal positive constant $C_p$.
    \begin{enumerate}
        \item If $p=1$,
        \begin{equation*}
    \norm{(T'_if)_{i \in \N}}_{L^{1, \infty}(\CN; \ell_2^{rc})} \leq C_1 \norm{f}_{L^1(\CN, \varphi)} ~ \text{for all } f \in L^1(\CN, \varphi);
\end{equation*}
\item if $1<p< \infty$
\begin{equation*}
    \norm{(T'_if)_{i \in \N}}_{L^p(\CN; \ell_2^{rc})} \leq C_p \norm{f}_{L^p(\CN, \varphi)} ~ \text{for all } f \in L^p(\CN, \varphi).
\end{equation*}
    \end{enumerate}
\end{thm}


The proof of \Cref{thm: quant main thm} is lengthy and is therefore divided into several steps. 
Our argument relies on Bourgain’s long--short variation decomposition technique, originally introduced in the study of variational inequalities \cite{bourgain1989pointwise}. 
This approach has subsequently been adapted and refined in various non-commutative and vector-valued settings; see, for instance, \cite{hong2021noncommutative}, \cite{hong2024quantitative}.

We begin by outlining the overall strategy of the proof.
Let $(r_i)_{i \in \mathbb{N}}$ be an increasing sequence of positive real numbers. 
For each $i \in \mathbb{N}$, define the interval
\[
I_i := [r_i, r_{i+1}).
\]
Each interval $I_i$ falls into exactly one of the following two cases:
\begin{enumerate}
    \item[{Case 1.}] $\delta^k \notin I_i$ for all $k \in \mathbb{N}$.
    \item[{Case 2.}] $\delta^k \in I_i$ for at least one $k \in \mathbb{N}$.
\end{enumerate}
If $I_i$ belongs to {Case 2}, we decompose it as
\[
I_i = [r_i, \delta^{k_i}) \,\cup\, [\delta^{k_i}, \delta^{l_i}) \,\cup\, [\delta^{l_i}, r_{i+1}),
\]
where
\[
k_i := \inf \{ k \in \mathbb{N} : \delta^k \in I_i \},
\qquad
l_i := \sup \{ k \in \mathbb{N} : \delta^k \in I_i \}.
\]
If there exists exactly one $k \in \mathbb{N}$ such that $\delta^k \in I_i$, we set $k_i = l_i$.
In this situation, we observe that for every $i \in \mathbb{N}$,
\begin{align*}
    \bigl\|(A'_{r_i}f - A'_{r_{i+1}}f)\bigr\|_{L^{1,\infty}(\mathcal{N}; \ell_2^{rc})}
    \leq {} &
    3 \bigl\| (A'_{r_i}f - A'_{\delta^{k_i}}f)\bigr\|_{L^{1,\infty}(\mathcal{N}; \ell_2^{rc})} \\
    &+ 3 \bigl\| (A'_{\delta^{k_i}}f - A'_{\delta^{l_i}}f) \bigr\|_{L^{1,\infty}(\mathcal{N}; \ell_2^{rc})} \\
    &+ 3 \bigl\| (A'_{\delta^{l_i}}f - A'_{r_{i+1}}f) \bigr\|_{L^{1,\infty}(\mathcal{N}; \ell_2^{rc})}.
\end{align*}

Next, we introduce two collections of intervals:
\begin{enumerate}
    \item $S$ denotes the collection of all intervals $I_i$ that either belong to {Case 1}, or are of the form $[r_i, \delta^{k_i})$ or $[\delta^{l_i}, r_{i+1})$.
    \item $L$ denotes the collection of all intervals $I_i$ of the form $[\delta^{k_i}, \delta^{l_i})$.
\end{enumerate}
It then follows immediately that
\begin{align*}
    \bigl\|(A'_{r_i}f - A'_{r_{i+1}}f)_{i \in \mathbb{N}}\bigr\|_{L^{1,\infty}(\mathcal{N}; \ell_2^{rc})}
    \leq {} &
    3 \bigl\|(A'_{r_i}f - A'_{r_{i+1}}f)_{i : I_i \in S}\bigr\|_{L^{1,\infty}(\mathcal{N}; \ell_2^{rc})} \\
    &+ 6 \bigl\|(A'_{r_i}f - A'_{r_{i+1}}f)_{i : I_i \in L}\bigr\|_{L^{1,\infty}(\mathcal{N}; \ell_2^{rc})}.
\end{align*}

Finally, consider the contribution from the long intervals. 
If $I_i \in L$, then we further estimate
\begin{align*}
    \bigl\|(A'_{r_i}f - A'_{r_{i+1}}f)_{i : I_i \in L}\bigr\|_{L^{1,\infty}(\mathcal{N}; \ell_2^{rc})}
    \leq {} &
    3 \bigl\|(A'_{\delta^{k_i}}f - \mathcal{E}_{k_i}f)_i\bigr\|_{L^{1,\infty}(\mathcal{N}; \ell_2^{rc})} \\
    &+ 3 \bigl\|(\mathcal{E}_{k_i}f - \mathcal{E}_{l_i}f)_i\bigr\|_{L^{1,\infty}(\mathcal{N}; \ell_2^{rc})} \\
    &+ 3 \bigl\|(A'_{\delta^{l_i}}f - \mathcal{E}_{l_i}f)_i\bigr\|_{L^{1,\infty}(\mathcal{N}; \ell_2^{rc})}.
\end{align*}


Let us now recall the following Burkholder-Guidy type inequalities for martingales from \cite[Theorem 2.1]{randrianantoanina2004square}.

\begin{thm}\label{thm: sq fn est mart-narcisse}
    There exists an absolute constant $K>0$ such that for any $L^1$-bounded martingale $x=(x_n)_{n \geq 1}$ there exists two sequences $y=(y_n)$ and $z= (z_n)$ such that 
    \begin{enumerate}
        \item[(i)] for every $n \in \N$, $x_n= y_n + z_n$.
        \item[(ii)]  $\norm{(dy_n)_{n \in\N}}_{L^{1, \infty}(\CM; \ell_2^{r})} + \norm{(dz_n)_{n \in \N}}_{L^{1, \infty}(\CN; \ell_2^{c})} \leq K \norm{x}_1$.
    \end{enumerate}
\end{thm}

Therefore, from \Cref{thm: sq fn est mart-narcisse} it follows that 
\begin{equation*}
    \norm{(\E_{k_i}f- E_{l_i}f)_{i}}_{L^{1, \infty}(\CN; \ell_2^{rc})} \leq C_p \norm{f}_{L^1(\CN)}
\end{equation*}
for some absolute constant $C_p>0$.

On the other hand, putting $w=1$ in \Cref{thm: main thm}, we obtain
\begin{equation*}
    \norm{(A'_{r_i}f- A'_{r_{i+1}}f)_{i: I_i \in L}}_{L^{1, \infty}(\CN; \ell_2^{rc})}
    \leq C_p \norm{f}_{L^1(\CN)}
\end{equation*}
for some absolute constant $C_p>0$. Therefore these arguments conclude the proof of \Cref{thm: quant main thm}, if for all $i \in \N$, the interval $I_i \in L$. Hence, it only remains to prove the theorem for the case when for all $i \in \N$, the interval $I_i \in S$.

\begin{thm}\label{thm: weak est quant}
    Let the collection $S$ and $(T_if)_{i \in S}$ be as defined above. Then there exists an absolute constant $C>0$ such that the following holds.
    \begin{equation*}
        \norm{(T'_if)_{i: I_i \in S}}_{L^{1, \infty}(\CN; \ell_2^{rc})} \leq C \norm{f}_{L^1(\CN, \varphi)} ~ \text{for all } f \in L^1(\CN, \varphi).
    \end{equation*}
\end{thm}

\subsection{Proof of \Cref{thm: weak est quant}}

    From \Cref{thm: CZ decomp}, we recall that
    \begin{equation*}
        f= g + b = g + b_d + b_{off}.
    \end{equation*}
     Let $(\epsilon_i)$ be a sequence of Rademacher variable on a Probability space $(\Omega,P)$ and define
    \begin{equation*}
        Tf(x):= \sum_{i \in S } \epsilon_i T_i'f(x).
    \end{equation*}
    By Khintchine's inequality, it is enough to prove that
    \begin{equation*}
        \norm{Th}_{L^{1, \infty}(L^\infty(\Omega) \otimes \CN)} \lesssim \norm{f}_{L^1(\CN, \varphi)} ~ \text{for } h \in \{g, b_d, b_{off}  \}.
    \end{equation*}
    \noindent \emph{\textbf{Estimate for bad part:}} First consider the projection $\zeta \in \CN$ and write
    \begin{equation*}
        Tb= (1-\zeta)Tb(1- \zeta) + \zeta Tb(1-\zeta) + (1-\zeta)Tb\zeta + \zeta Tb \zeta.
    \end{equation*}
    Further, by using  \cref{equi} and   \Cref{lem: zeta estimate}, observe that
    \begin{align*}
        \tilde{\varphi}(\{\abs{Tb} > \lambda/2 \}) &\lesssim \varphi(1- \zeta) + \tilde{\varphi} (\abs{\zeta Tb \zeta } > \lambda/8)\\
        &\lesssim \frac{\norm{f}_1}{\lambda} + \tilde{\varphi} (\abs{\zeta Tb \zeta } > \lambda/8).
    \end{align*}
    Therefore, we only need to show 
    \begin{equation*}
        \tilde{\varphi} (\abs{\zeta Tb \zeta } > \lambda/8) \lesssim \frac{\norm{f}_1}{\lambda}.
    \end{equation*}
    Now, since $\sum_n \norm{p_n}_1 \lesssim \frac{\norm{f}_1}{\lambda}$, by Chebyshev's inequality, it is enough to show that
    \begin{equation*}
        \norm{\zeta Tb \zeta}_{L^1(L^\infty(\Omega) \otimes \CN)} \lesssim \lambda \sum_n \norm{p_n}_1.
    \end{equation*}
    At this point, for all $k \in \Z$, let us define the set of indices $S_k:= \{i \in S: [r_i, r_{i+1}) \subset [\delta^k, \delta^{k+1})  \}$. Therefore, it is reduced to show that 
    \begin{equation*}
        \sum_{k } \sum_{i \in S_k} \norm{\zeta T'_ib \zeta}_{L^1(\CN)} \lesssim \lambda \sum_n \norm{p_n}_1.
    \end{equation*}
    Let us now write, $b= \sum_n b_n$, where, $b_n:= p_n(f-f_n)p_n + p_n(f-f_n)q_n + q_n(f-f_n)p_n$.

       \begin{lem}
        Let $k \in \Z$ and $i \in S_k$. Then for all $n \geq k$
        \begin{align*}
            \zeta(x) (A'_{r_i} - A'_{r_{i+1}}) b_n(x) \zeta(x)=0, ~ \text{for all } x \in G. 
        \end{align*}
    \end{lem}

    \begin{proof}
        Since $i \in S_k$, we have $[r_i, r_{i+1}) \subset [\delta^k, \delta^{k+1})$. Let $x \in G$ and $d(y,x)< r_{i+1}$. Therefore, $d(y,x)< \delta^{n+1}$, since $n \geq k$. Further, consider the unique dyadic set $Q^n_{\alpha(x)}$ containing $x$ for some $\alpha(x) \in I_n$. Hence, we also have $d(x, z^n_{\alpha(x)})< C_0 \delta^n< C_1 \delta^{n+1}$. So, by triangle inequality, we obtain $d(y, z^n_{\alpha(x)} )< (3C_1 +1)\delta^{n+1}$ and we found that $B(x, r_{i+1}) \subseteq \tilde{Q}^n_{\alpha(x)}$. Now, by $(3)$ of \Cref{lem: zeta estimate} observe that 
        \begin{align*}
            \zeta(x) A'_{r_{i+1}} b_n(x) \zeta(x)
            = \zeta(x) \frac{1}{\mu(B_{r_{i+1}})} \int_{B(x, r_{i+1})} b_n(y) \chi_{G \setminus \tilde{Q}^n_{\alpha(x)}}(y) d \mu(y) \zeta(x)=0.
        \end{align*}
        Similarly, we can prove that $\zeta(x) A'_{r_{i}} b_n(x) \zeta(x)=0$.
    \end{proof}
In view of the above lemma, we are left to show that 
    \begin{equation}\label{final-1}
        \sum_{k } \sum_{i \in S_k} \sum_{n: n<k} \norm{\zeta T'_ib_n \zeta}_{L^1(\CN)}= \sum_{k } \sum_{n: n<k} \sum_{i \in S_k} \norm{\zeta T'_ib_n \zeta}_{L^1(\CN)}  \lesssim \lambda \sum_n \norm{p_n}_1.
    \end{equation}
First, observe the following:
    \begin{align*}
        A'_{r_i} b_n(x) - A'_{r_{i+1}} b_n(x)= A'_{r_i} b_n(x) -\frac{1}{\mu(B_{r_{i+1}})} \int_{B(x, r_{i})} b_n(y) d\mu(y) 
        &+ \frac{1}{\mu(B_{r_{i+1}})} \int_{B(x, r_{i})} b_n(y) d\mu(y) \\
        &- A'_{r_{i+1}} b_n(x).
    \end{align*}
    Therefore, by the triangle inequality, we obtain
    \begin{align*}
        \sum_{i \in S_k} \norm{\zeta (A'_{r_i} - A'_{r_{i+1}}) b_n \zeta}_1
        &\leq
        \sum_{i \in S_k} \int_G \norm{\frac{1}{\mu(B_{r_{i+1}})} \int_{B(x, r_{i+1}) \setminus B(x, r_i)} b_n(y) d \mu(y) }_1 d \mu(x) +\\
        & 
        \sum_{i \in S_k} \left(\frac{1}{\mu(B_{r_{i}})} - \frac{1}{\mu(B_{r_{i+1}})} \right) \int_G \norm{\int_{B(x, r_{i})} b_n(y) d\mu(y) }_1 d\mu(x) \\
        &= I^1_{k,n} + I^2_{k,n}, 
    \end{align*}
where 
\begin{align*}
(i)~~ ~&I^1_{k,n}= \sum_{i \in S_k} \int_G \norm{\frac{1}{\mu(B_{r_{i+1}})} \int_{B(x, r_{i+1}) \setminus B(x, r_i)} b_n(y) d \mu(y) }_1 d \mu(x)~ \text{ and }\\
(ii)~~ ~&I^2_{k,n}=\sum_{i \in S_k} \left(\frac{1}{\mu(B_{r_{i}})} - \frac{1}{\mu(B_{r_{i+1}})} \right) \int_G \norm{\int_{B(x, r_{i})} b_n(y) d\mu(y) }_1 d\mu(x).
\end{align*}
Although the proof of the following lemma is similar to that of \Cref{lem: norm est prop}, note that the statement differs technically. For the reader’s convenience, we therefore include a complete proof.
\begin{lem}
    Suppose  $k>n \geq n_{r_0}$. Fix $y \in G$ and set $ A(y):= B(y, r + 2C_1 \delta^n) \setminus B(y, r- 2C_1 \delta^n)$, where $\delta^k \leq r< \delta^{k+1}$. Then the following holds. 
\begin{align*}
    \int_G \frac{\chi_{A(y)}(x)}{\mu(B(x,\delta^k))} d\mu(x) \lesssim  \delta^{(n-k)\epsilon}.
\end{align*}
\end{lem}

\begin{proof}
    Note that for any $\epsilon \in (0,1)$ we can find $\{u_i: 1 \leq i \leq M \} \subseteq B(y, r + C_1 \delta^n)$, where $M \leq D$ such that 
    \[
    B(y, r + C_1 \delta^n) \subseteq \bigcup_{i=1}^M B(u_i, r + C_1 \delta^n)
    \]
    and if we fix $x \in B(y, r + C_1 \delta^n)$ with $j$ be the first index such that $x \in B(u_j, r + C_1 \delta^n)$, then 
    \begin{align*}
        \frac{\chi_{B(y, r + C_1 \delta^n)}(x)}{\mu(B(x, r + C_1 \delta^n))}
        &\leq (1+\epsilon) \frac{\chi_{B(y, r + C_1 \delta^n) \cap B(u_j, r + C_1 \delta^n)}(x)}{\mu(B(u_j, r + C_1 \delta^n))}\\
        &\leq (1+\epsilon) \sum_{i=1}^M \frac{\chi_{B(y, r + C_1 \delta^n) \cap B(u_i, r + C_1 \delta^n)}(x)}{\mu(B(u_i, r + C_1 \delta^n))}\\
         &\leq (1+\epsilon) \sum_{i=1}^M \frac{\chi_{B(u_i, r + C_1 \delta^n)}(x)}{\mu(B(u_i, r))}.
    \end{align*}

    Finally,
    \begin{align*}
        &\int_G \frac{\chi_{A(y)}(x)}{\mu(B(x,\delta^k))} d\mu(x) \\
        &\leq (K+1)(C_1 +\delta)^\epsilon \int_G \frac{\chi_{A(y)}(x)}{\mu(B(x, r + C_1 \delta^n))} \chi_{B(y, r + C_1 \delta^n)}(x) d\mu(x) \\
        &\leq (K+1)(C_1 +\delta)^\epsilon (1+\epsilon) \int_G \sum_{i=1}^M \frac{\chi_{A(y) \cap B(u_i, r + C_1 \delta^n)}(x)}{\mu(B(u_i, r))} d \mu(x)\\
        & \leq (K+1)(C_1 +\delta)^\epsilon (1+\epsilon) \int_G \sum_{i=1}^M \frac{\chi_{A(y)}(x)}{\mu(B(y, r))} \frac{\mu(B(y, r))}{\mu(B(u_i, r)) } d\mu(x)\\
        & \leq (K+1)(C_1 +\delta)^\epsilon (1+\epsilon) \int_G \sum_{i=1}^M \frac{\chi_{A(y)}(x)}{\mu(B(y, r))} \frac{\mu(B(u_i, 2r + C_1 \delta^n))}{\mu(B(u_i, r)) } d\mu(x)\\
        & \qquad \qquad \text{(since $B(y, r) \subseteq B(u_i, 2r + C_1 \delta^n)$)}\\
        &\leq (K+1)^2(C_1 +\delta)^\epsilon (1+\epsilon)(2+C_1)^\epsilon \sum_{i=1}^M \frac{\mu(A(y))}{\mu(B(y, r))}\\
        &\leq (K+1)^2(C_1 +\delta)^\epsilon (1+\epsilon)(2+C_1)^\epsilon M K_\epsilon (\delta^{n-k})^\epsilon
        \lesssim \delta^{(n-k)\epsilon}.
    \end{align*}
    Therefore, the result follows.
\end{proof}

Now we note the following facts as a proposition. 
\begin{prop}
    Let $n<k$ and $\alpha \in I_n$. Consider $Q^n_\alpha$, then there exists $l_\alpha \in [\delta^k, \delta^{k+1})$ such that
    \begin{equation*}
        \cup_{i \in S_k} \{ Q^n_\alpha \cap (B(x, r_{i+1}) \setminus B(x, r_{i})): Q^n_\alpha \cap \partial (B(x, r_{i+1}) \setminus B(x, r_{i})) \neq \emptyset \}  \subseteq B(x, l_\alpha + 2C_1 \delta^n) \setminus B(x, l_\alpha- 2C_1 \delta^n)
    \end{equation*}
    and 
    \begin{equation*}
        \cup_{i \in S_k} \{ Q^n_\alpha \cap B(x, r_{i}): Q^n_\alpha \cap \partial (B(x, r_{i})) \neq \emptyset \}  \subseteq B(x, l_\alpha + 2C_1 \delta^n) \setminus B(x, l_\alpha- 2C_1 \delta^n).
    \end{equation*}
\end{prop}

 \begin{proof}
 Take  a $r_i$  such that 
  $$Q^n_\alpha \cap \partial (B(x, r_{i+1}) \setminus B(x, r_{i})) \neq \emptyset $$
Now if $  Q^n_\alpha \cap \partial (B(x, r_{i+1})) \neq \emptyset $, 
then take $l_\alpha$ to be $r_{ i+1}$ and if $ Q^n_\alpha \cap \partial (B(x, r_{i})) \neq \emptyset $, then we take $ l_\alpha$ to be $ r_i$. 
\end{proof}

Let $F$ be a measurable subset of $G$. Similar to the previous section, for every $n \in \Z$, let us define
\[
\CI(F,n):= \bigcup_{\alpha \in I_n,~\partial(F) \cap Q^n_\alpha \neq \emptyset} Q^n_\alpha \cap F,
\]

\begin{lem}\label{lem: bad part est prop}
Let $s \in G$ and $ k>n >n_{r_0}$. Then for any measurable subset $F$ of $G$, the following estimates hold.
    \begin{align*}
        \norm{\int_{\CI(F,n)} b_n^{off} (t) d\mu(t)}_1
\lesssim \lambda \int_{\CI(F,n)} \tau(p_n(t)) d \mu(t)
    \end{align*}
\end{lem}

\begin{proof}
   Recall that   $b_n^{off}= p_n(f- f_n) q_n + q_n(f- f_n) p_n
        =p_nfq_n + q_nfp_n$, (see \Cref{thm: Cuculescu}(1)), it is enough to prove the result for $b_n^{off}= p_nfq_n$. First, applying \Cref{prop:Holder-type-ineq} with $ \frac{1}{1} = \frac{1}{1}+ \frac{1}{\infty}$, we note that
\begin{align*}
    \norm{\int_{F \cap Q^n_\alpha} (p_nfq_n) (t) d\mu(t)}_1
   & \leq \norm{\int_{ Q^n_\alpha} \chi_{F}(t) p_{Q^n_\alpha}f(t)q_{Q^n_\alpha} d\mu(t)}_1\\
    &\leq \norm{\left(\int_{Q^n_\alpha} \chi_{F}(t) p_{Q^n_\alpha}f(t)p_{Q^n_\alpha} \right)^{1/2} }_1
    \norm{\left(\int_{ Q^n_\alpha} \chi_{F}(t) q_{Q^n_\alpha}f(t)q_{Q^n_\alpha} \right)^{1/2} }_\infty\\
    &=\norm{\int_{Q^n_\alpha} \chi_{F}(t) p_{Q^n_\alpha}f(t)p_{Q^n_\alpha}  }_{1/2}^{1/2}
    \norm{\int_{ Q^n_\alpha} \chi_{F}(t) q_{Q^n_\alpha}f(t)q_{Q^n_\alpha}  }^{1/2}_\infty.
\end{align*}
Also,
\begin{align*}
    \norm{\int_{Q^n_\alpha} \chi_{F}(t) p_{Q^n_\alpha}f(t)p_{Q^n_\alpha}  }_{1/2}
    &=\norm{p_{Q^n_\alpha} \int_{Q^n_\alpha} \chi_{F}(t) p_{Q^n_\alpha}f(t)p_{Q^n_\alpha}  }_{1/2}\\
    &\leq \norm{p_{Q^n_\alpha}}_1  \norm{\int_{Q^n_\alpha \cap \chi_{F}} p_{Q^n_\alpha}f(t)p_{Q^n_\alpha}}_1~ \text{by Holder's inequality with } \frac{1}{1/2}= \frac{1}{1} + \frac{1}{1}.
\end{align*}
Now since $q_{Q^n_\alpha}f(t)q_{Q^n_\alpha} \lesssim \lambda $ and $p_{Q^n_\alpha}f(t)p_{Q^n_\alpha} \lesssim \lambda p_{Q^n_\alpha}$, we obtain
\begin{align*}
    \norm{\int_{F \cap Q^n_\alpha} (p_nfq_n) (t) d\mu(t)}_1
    &\lesssim \norm{p_{Q^n_\alpha}}^{1/2}_1  \norm{\int_{Q^n_\alpha \cap \chi_{F}} p_{Q^n_\alpha}f(t)p_{Q^n_\alpha}}^{1/2}_1 (\lambda \mu(Q^n_\alpha \cap \chi_{F})^{1/2}\\
    &\lesssim \tau(p_{Q^n_\alpha})^{1/2} \left( \lambda \tau(p_{Q^n_\alpha }) \mu(Q^n_\alpha \cap \chi_{F} \right)^{1/2} (\lambda \mu(Q^n_\alpha \cap \chi_{F})^{1/2}\\
    &= \lambda \tau(p_{Q^n_\alpha}) \mu(Q^n_\alpha \cap F)\\
    &= \lambda \varphi(p_n \chi_{Q^n_\alpha \cap F}).
\end{align*}
Furthermore,
\begin{align*}
    \norm{\int_{\CI(F,n)} (p_nfq_n) (t) d\mu(t)}_1
    &\leq  \sum_{\alpha \in I_n,~ \partial(F) \cap Q^n_\alpha \neq \emptyset} \norm{\int_{B(s, \delta^k) \cap Q^n_\alpha} (p_nfq_n) (t) d\mu(t)}_1\\
    &\lesssim \lambda \sum_{\alpha \in I_n,~ \partial(F) \cap Q^n_\alpha \neq \emptyset} \varphi(p_n \chi_{Q^n_\alpha \cap F})\\
    &= \lambda \varphi(p_n \chi_{\CI(F,n)})\\
    &=\lambda \int_{\CI(F,n)} \tau(p_n(t)) d \mu(t).
\end{align*}
\end{proof}

\begin{prop} With the above notations, we have the following:
\begin{align*}
 &(i)~~~I^1_{k,n} \lesssim \lambda \delta^{(n-k)\epsilon} \varphi(p_n) \text{ and }\\
&(ii)~~~I^2_{k,n} \lesssim \lambda \delta^{(n-k)\epsilon} \varphi(p_n).
    \end{align*} 
\end{prop}

\begin{proof} \emph{ of $(i)$}. 
    Let $g_n \in \{ p_nfp_n, p_nfq_n, q_nfp_n \}$. Write $A(x,i):= B(x, r_{i+1}) \setminus B(x, r_i) $ and  observe that 
    \begin{align*}
        \norm{ \int_{ \CI (A(x,i),n )} g_n(y)d \mu(y) }_1 \lesssim \lambda \sum_{\alpha \in I_n,~ \partial(A(x,i)) \cap Q^n_\alpha \neq \emptyset} \varphi \left( p_n \chi_{Q^n_\alpha \cap A(x,i)} \right).
    \end{align*}
    Therefore,
    \begin{align*}
        &\sum_{i \in S_k} \int_G \frac{1}{\mu(B(x, r_{i+1}))} \norm{ \int_{ \CI (A(x,i),n )} g_n(y)d \mu(y) }_1 d \mu(x) \\
        &\lesssim  \lambda \sum_{i \in S_k} \int_G \frac{1}{\mu(B(x, r_{i+1}))} \left( \sum_{\alpha \in I_n,~ \partial(A(x,i)) \cap Q^n_\alpha \neq \emptyset} \varphi \left( p_n \chi_{Q^n_\alpha \cap A(x,i)} \right) \right) d \mu(x)\\
        &\leq \lambda \int_G \frac{1}{\mu(B(x, \delta^k))}  \left( \sum_{i \in S_k} \sum_{\alpha \in I_n, \partial(A(x,i)) \cap Q^n_\alpha \neq \emptyset}   \varphi ( p_n \chi_{Q^n_\alpha \cap A(x,i)}) \right) d\mu(x) \\
        &= \lambda \int_G \frac{1}{\mu(B(x, \delta^k))} \left( \int_G  \sum_{\alpha \in I_n}  \sum_{i \in S_k: \partial(A(x,i)) \cap Q^n_\alpha \neq \emptyset} \tau (p_n(t) )   \chi_{Q^n_\alpha \cap A(x,i)}(t) d\mu(t)  \right) d\mu(x)\\
        &= \lambda \int_G \frac{1}{\mu(B(x, \delta^k))} \left( \int_G \sum_{\alpha \in I_n} \tau(p_{Q^n_\alpha} \chi_{Q^n_\alpha}(t)) \chi_{B(x, l_\alpha+ 2C_1\delta^n) \setminus B(x, l_\alpha- 2C_1\delta^n)} (t) \right) d\mu(x)\\
        &= \lambda \int_G \frac{1}{\mu(B(x, \delta^k))} \left(  \sum_{\alpha \in I_n} \int_{B(x, l_\alpha+ 2C_1\delta^n) \setminus B(x, l_\alpha- 2C_1\delta^n)} \tau(p_{Q^n_\alpha} \chi_{Q^n_\alpha}(t)) \right) d\mu(x)\\
        &= \lambda \sum_{\alpha \in I_n} \int_G \frac{1}{\mu(B(x, \delta^k))} \left( \int_{B(x, l_\alpha+ 2C_1\delta^n) \setminus B(x, l_\alpha- 2C_1\delta^n)} \tau(p_{Q^n_\alpha} \chi_{Q^n_\alpha}(t)) \right) d\mu(x)\\
        &=\lambda \int_G \sum_{\alpha \in I_n} \tau(p_{Q^n_\alpha} \chi_{Q^n_\alpha}(t)) \left( \int_G \frac{\chi_{B(t, l_\alpha+ 2C_1\delta^n) \setminus B(t, l_\alpha- 2C_1\delta^n)}(x)}{\mu(B(x, \delta^k))} d\mu(x) \right) d\mu(t)\\
        &\lesssim \lambda \delta^{(n-k)\epsilon} \int_G \tau(p_n(t)) d\mu(t)
        = \lambda \delta^{(n-k)\epsilon} \varphi(p_n).\\
    \end{align*}

\emph{For the proof of $(ii)$}, 
    let $g_n \in \{ p_nfp_n, p_nfq_n, q_nfp_n \}$. Now  observe that 
    \begin{align*}
        \norm{ \int_{ \CI (B(x, r_i),n )} g_n(y)d \mu(y) }_1 \lesssim \lambda \sum_{\alpha \in I_n,~ \partial(B(x, r_i)) \cap Q^n_\alpha \neq \emptyset} \varphi \left( p_n \chi_{Q^n_\alpha \cap B(x, r_i)} \right).
    \end{align*}
    Therefore,
    \begin{align*}
        &\sum_{i \in S_k} \int_G \left(\frac{1}{\mu(B_{r_{i}})} - \frac{1}{\mu(B_{r_{i+1}})} \right) \norm{ \int_{ \CI (B(x, r_i),n )} g_n(y)d \mu(y) }_1 d\mu(x)\\
        &\lesssim \lambda \delta^\epsilon \sum_{i \in S_k} \int_G \frac{1}{\mu(B(x, \delta^k))} \left( \sum_{\alpha \in I_n,~ \partial(B(x, r_i)) \cap Q^n_\alpha \neq \emptyset} \varphi \left( p_n \chi_{Q^n_\alpha \cap B(x, r_i)} \right) \right) d\mu(x)\\
        &\lesssim \lambda \int_G \frac{1}{\mu(B(x, \delta^k))}  \left( \sum_{i \in S_k} \sum_{\alpha \in I_n, \partial(B(x, r_i)) \cap Q^n_\alpha \neq \emptyset}   \varphi ( p_n \chi_{Q^n_\alpha \cap B(x, r_i)}) \right) d\mu(x) \\
        &= \lambda \int_G \frac{1}{\mu(B(x, \delta^k))} \left( \int_G  \sum_{\alpha \in I_n}  \sum_{i \in S_k: \partial(B(x, r_i)) \cap Q^n_\alpha \neq \emptyset} \tau (p_n(t) )   \chi_{Q^n_\alpha \cap B(x, r_i)}(t) d\mu(t)  \right) d\mu(x)\\
        &= \lambda \int_G \frac{1}{\mu(B(x, \delta^k))} \left( \int_G \sum_{\alpha \in I_n} \tau(p_{Q^n_\alpha} \chi_{Q^n_\alpha}(t)) \chi_{B(x, l_\alpha+ 2C_1\delta^n) \setminus B(x, l_\alpha- 2C_1\delta^n)} (t) \right) d\mu(x)\\
        &= \lambda \int_G \frac{1}{\mu(B(x, \delta^k))} \left(  \sum_{\alpha \in I_n} \int_{B(x, l_\alpha+ 2C_1\delta^n) \setminus B(x, l_\alpha- 2C_1\delta^n)} \tau(p_{Q^n_\alpha} \chi_{Q^n_\alpha}(t)) \right) d\mu(x)\\
        &= \lambda \sum_{\alpha \in I_n} \int_G \frac{1}{\mu(B(x, \delta^k))} \left( \int_{B(x, l_\alpha+ 2C_1\delta^n) \setminus B(x, l_\alpha- 2C_1\delta^n)} \tau(p_{Q^n_\alpha} \chi_{Q^n_\alpha}(t)) \right) d\mu(x)\\
        &=\lambda \int_G \sum_{\alpha \in I_n} \tau(p_{Q^n_\alpha} \chi_{Q^n_\alpha}(t)) \left( \int_G \frac{\chi_{B(t, l_\alpha+ 2C_1\delta^n) \setminus B(t, l_\alpha- 2C_1\delta^n)}(x)}{\mu(B(x, \delta^k))} d\mu(x) \right) d\mu(t)\\
        &\lesssim \lambda \delta^{(n-k)\epsilon} \int_G \tau(p_n(t)) d\mu(t)
        = \lambda \delta^{(n-k)\epsilon} \varphi(p_n).
    \end{align*}

\end{proof}

Now we are ready to establish the claim in \cref{final-1}. Indeed, we observe that 
\begin{align*}
    \sum_{k } \sum_{i \in S_k} \sum_{n: n<k} \norm{\zeta T'_ib_n \zeta}_{L^1(\CN)}
         &=\sum_{k } \sum_{n: n<k} \sum_{i \in S_k} \norm{\zeta T'_ib_n \zeta}_{L^1(\CN)}\\
        &\lesssim \sum_{n} \sum_{k>n} \lambda   \delta^{(n-k)\epsilon} \varphi(p_n)\\
    &\leq \lambda   \sum_n \varphi(p_n)\\
    &= \lambda  \varphi(1-q) \\ 
    &\lesssim  \norm{f}_{1}.
    \end{align*}
This establishes the claim in \cref{final-1}. Consequently, the proof of part~$(1)$ of \Cref{thm: quant main thm} is complete. Moreover, part~$(2)$ of \Cref{thm: quant main thm} follows verbatim from the arguments presented in \Cref{strongbdd}.\\

 \noindent\textbf{Acknowledgements}: 
 P. Bikram acknowledges the support of the grant ANRF/ARGM/2025/001021/MTR, 
 Government of India.


\begin{thebibliography}{GJOW22}
	
	\bibitem[Bou89]{bourgain1989pointwise}
	Jean Bourgain, \emph{Pointwise ergodic theorems for arithmetic sets},
	Publications Math{\'e}matiques de l'IH{\'E}S \textbf{69} (1989), 5--41.
	
	\bibitem[Cad19]{cadilhac2019noncommutative}
	L{\'e}onard Cadilhac, \emph{Noncommutative khintchine inequalities in
		interpolation spaces of lp-spaces}, Advances in Mathematics \textbf{352}
	(2019), 265--296.
	
	\bibitem[CCAP22]{cadilhac2022spectral}
	L{\'e}onard Cadilhac, Jos{\'e}~M Conde-Alonso, and Javier Parcet,
	\emph{Spectral multipliers in group algebras and noncommutative
		calder{\'o}n-zygmund theory}, Journal de Math{\'e}matiques Pures et
	Appliqu{\'e}es \textbf{163} (2022), 450--472.
	
	\bibitem[Cuc71]{cuculescu1971martingales}
	Ioan Cuculescu, \emph{Martingales on von neumann algebras}, Journal of
	Multivariate Analysis \textbf{1} (1971), no.~1, 17--27.
	
	\bibitem[CW22]{cadilhac2022noncommutative}
	L{\'e}onard Cadilhac and Simeng Wang, \emph{Noncommutative maximal ergodic
		inequalities for amenable groups}, arXiv preprint arXiv:2206.12228 (2022).
	
	\bibitem[FF12]{fefferfeffer12}
	Charles Fefferman and Robert Fefferman, \emph{Princeton lectures in analysis
		[book reviews of mr1970295, mr1976398, mr2129625, mr2827930]}, Notices Amer.
	Math. Soc. \textbf{59} (2012), no.~5, 641--647, With contributions from Paul
	Hagelstein, Nata\v sa Pavlovi\'c{} and Lillian Pierce. \MR{2954291}
	
	\bibitem[FJWZ25]{fan2025weighted}
	Wenfei Fan, Yong Jiao, Lian Wu, and Dejian Zhou, \emph{Weighted norm estimates
		of noncommutative calder$\backslash$'$\{$o$\}$ n-zygmund operators}, arXiv
	preprint arXiv:2501.04951 (2025).
	
	\bibitem[GJOW22]{galkazka2022sharp}
	Tomasz Galazka, Yong Jiao, Adam Osekowski, and Lian Wu, \emph{The sharp
		weighted maximal inequalities for noncommutative martingales}, arXiv preprint
	arXiv:2211.09358 (2022).
	
	\bibitem[HL21]{hong2021quantitative}
	Guixiang Hong and Wei Liu, \emph{Quantitative ergodic theorems for actions of
		groups of polynomial growth}, arXiv preprint arXiv:2104.02635 (2021).
	
	\bibitem[HL26]{HONG_LIU_2026}
	GUIXIANG HONG and WEI LIU, \emph{Quantitative ergodic theorems for actions of
		groups of polynomial growth}, Ergodic Theory and Dynamical Systems
	\textbf{46} (2026), no.~1, 128–181.
	
	\bibitem[HLRX22]{hong2022noncommutative}
	Guixiang Hong, Xudong Lai, Samya~Kumar Ray, and Bang Xu, \emph{Noncommutative
		maximal strong $ l\_p $ estimates of calder$\backslash$'on-zygmund
		operators}, arXiv preprint arXiv:2212.13150 (2022).
	
	\bibitem[HLW21a]{hong2021noncommutative}
	Guixiang Hong, Ben Liao, and Simeng Wang, \emph{Noncommutative maximal ergodic
		inequalities associated with doubling conditions}.
	
	\bibitem[HLW21b]{Hongliaowang21}
	\bysame, \emph{Noncommutative maximal ergodic inequalities associated with
		doubling conditions}, Duke Math. J. \textbf{170} (2021), no.~2, 205--246.
	\MR{4202493}
	
	\bibitem[HLX23]{Homglaixu23}
	Guixiang Hong, Xudong Lai, and Bang Xu, \emph{Maximal singular integral
		operators acting on noncommutative {$L_p$}-spaces}, Math. Ann. \textbf{386}
	(2023), no.~1-2, 375--414. \MR{4585152}
	
	\bibitem[HLX24]{hong2024quantitative}
	Guixiang Hong, Wei Liu, and Bang Xu, \emph{Quantitative mean ergodic
		inequalities: Power bounded operators acting on one single noncommutative lp
		space}, Journal of Functional Analysis \textbf{286} (2024), no.~1, 110190.
	
	\bibitem[HM17]{hong2017vector}
	Guixiang Hong and Tao Ma, \emph{Vector valued q-variation for differential
		operators and semigroups i}, Mathematische Zeitschrift \textbf{286} (2017),
	no.~1, 89--120.
	
	\bibitem[HRW23]{Hongraywang23}
	Guixiang Hong, Samya~Kumar Ray, and Simeng Wang, \emph{Maximal ergodic
		inequalities for some positive operators on noncommutative {$L_p$}-spaces},
	J. Lond. Math. Soc. (2) \textbf{108} (2023), no.~1, 362--408. \MR{4611833}
	
	\bibitem[Hyt12]{Hytonen12ann}
	Tuomas~P. Hyt\"onen, \emph{The sharp weighted bound for general
		{C}alder\'on-{Z}ygmund operators}, Ann. of Math. (2) \textbf{175} (2012),
	no.~3, 1473--1506. \MR{2912709}
	
	\bibitem[JKRW98]{jones1998oscillation}
	Roger~L Jones, Robert Kaufman, Joseph~M Rosenblatt, and M{\'a}t{\'e} Wierdl,
	\emph{Oscillation in ergodic theory}, Ergodic Theory and Dynamical Systems
	\textbf{18} (1998), no.~4, 889--935.
	
	\bibitem[JOR96]{jones1996square}
	Roger~L Jones, Iosif~V Ostrovskii, and Joseph~M Rosenblatt, \emph{Square
		functions in ergodic theory}, Ergodic theory and Dynamical systems
	\textbf{16} (1996), no.~2, 267--305.
	
	\bibitem[JRW03]{jones2003oscillation}
	Roger~L Jones, Joseph~M Rosenblatt, and M{\'a}t{\'e} Wierdl, \emph{Oscillation
		in ergodic theory: higher dimensional results}, Israel Journal of Mathematics
	\textbf{135} (2003), no.~1, 1--27.
	
	\bibitem[JX07a]{Junge2007}
	Marius Junge and Quanhua Xu, \emph{Noncommutative maximal ergodic theorems},
	Journal of the American Mathematical Society \textbf{20} (2007), no.~2,
	385--439. \MR{2276775}
	
	\bibitem[JX07b]{junge2007noncommutative}
	\bysame, \emph{Noncommutative maximal ergodic theorems}, Journal of the
	American Mathematical Society \textbf{20} (2007), no.~2, 385--439.
	
	\bibitem[KM22]{kurki2022extension}
	Emma-Karoliina Kurki and Carlos Mudarra, \emph{On the extension of muckenhoupt
		weights in metric spaces}, Nonlinear Analysis \textbf{215} (2022), 112671.
	
	\bibitem[LPP91]{lust1991non}
	Fran{\c{c}}oise Lust-Piquard and Gilles Pisier, \emph{Non commutative
		khintchine and paley inequalities}, Arkiv f{\"o}r matematik \textbf{29}
	(1991), 241--260.
	
	\bibitem[Mei07]{Mei07}
	Tao Mei, \emph{Operator valued {H}ardy spaces}, Mem. Amer. Math. Soc.
	\textbf{188} (2007), no.~881, vi+64. \MR{2327840}
	
	\bibitem[Muc72]{muckenhoupt1972weighted}
	Benjamin Muckenhoupt, \emph{Weighted norm inequalities for the hardy maximal
		function}, Transactions of the American Mathematical Society \textbf{165}
	(1972), 207--226.
	
	\bibitem[Par09]{parcet2009pseudo}
	Javier Parcet, \emph{Pseudo-localization of singular integrals and
		noncommutative calder{\'o}n--zygmund theory}, Journal of Functional Analysis
	\textbf{256} (2009), no.~2, 509--593.
	
	\bibitem[PX03]{pisier2003non}
	Gilles Pisier and Quanhua Xu, \emph{Non-commutative lp-spaces}, Handbook of the
	geometry of Banach spaces, vol.~2, Elsevier, 2003, pp.~1459--1517.
	
	\bibitem[Ran04]{randrianantoanina2004square}
	Narcisse Randrianantoanina, \emph{Square function inequalities for
		non-commutative martingales}, Israel Journal of Mathematics \textbf{140}
	(2004), no.~1, 333--365.
	
	\bibitem[Rie41]{riesz1941another}
	Frederick Riesz, \emph{Another proof of the mean ergodic theorem}, Acta Univ.
	Szeged. Sect. Sci. Math \textbf{10} (1941), 75--76.
	
	\bibitem[RS24]{saha-rayweighted}
	Samya~Kumar Ray and Diptesh Saha, \emph{Weighted weak $(1, 1) $ estimate for
		non-commutative square function}, arXiv preprint arXiv:2410.00987, To appear
	in Bul. of the London Mathematical Society (2024).
	
	\bibitem[TH12]{TuomasHytonen2012}
	Anna~Kairema Tuomas~Hytönen, \emph{Systems of dyadic cubes in a doubling
		metric space}, Colloquium Mathematicae \textbf{126} (2012), no.~1, 1--33
	(eng).
	
\end{thebibliography}

\providecommand{\bysame}{\leavevmode\hbox to3em{\hrulefill}\thinspace}
\providecommand{\MR}{\relax\ifhmode\unskip\space\fi MR }
\providecommand{\MRhref}[2]{%
	\href{http://www.ams.org/mathscinet-getitem?mr=#1}{#2}
}
\providecommand{\href}[2]{#2}

\end{document}